\documentclass[11pt]{amsart}
\usepackage[margin=1in]{geometry}
\usepackage{amsmath,amssymb,mathtools}
\usepackage{microtype}
\usepackage[T1]{fontenc}
\usepackage{lmodern}
\usepackage[hidelinks]{hyperref}
\hypersetup{
  pdftitle={Endpoint eigenfunction restriction estimates in codimension two},
  pdfauthor={Xing Wang and Cheng Zhang}
}
\numberwithin{equation}{section}
\newtheorem{mainthm}{Theorem}
\newtheorem{theorem}{Theorem}[section]
\newtheorem{proposition}[theorem]{Proposition}
\newtheorem{lemma}[theorem]{Lemma}

\theoremstyle{remark}
\newtheorem{remark}[theorem]{Remark}
\newcommand{\R}{\mathbb R}
\newcommand{\N}{\mathbb N}
\newcommand{\Sph}{\mathbb S}
\newcommand{\HH}{\mathcal H}

\newcommand{\ind}{\mathbf 1}
\newcommand{\pv}{\operatorname{p.v.}}
\newcommand{\supp}{\operatorname{supp}}
\newcommand{\W}{\mathcal W_\lambda}
\newcommand{\RR}{\mathcal R_\Sigma}
\newcommand{\sgn}{\operatorname{sgn}}
\newcommand{\dd}{\,d}
\newcommand{\wt}{\widetilde}
\newcommand{\norm}[1]{\left\lVert #1\right\rVert}

\newcommand{\ip}[2]{\left\langle #1,#2\right\rangle}
\allowdisplaybreaks[2]
\title[Endpoint eigenfunction restriction]{Endpoint eigenfunction restriction estimates in codimension two}
\date{}
\author{Xing Wang and Cheng Zhang}
\address{School of Mathematics, Hunan University, Changsha, HN 410012, China}
\email{xingwang@hnu.edu.cn}
\address{Yau Mathematical Sciences Center,
	Tsinghua University,
	Beijing, BJ 100084, China}
\email{czhang98@tsinghua.edu.cn}
\begin{document}
	\begin{abstract}
		We investigate the optimal endpoint $L^2$ restriction estimates of Laplace eigenfunctions on submanifolds of codimension 2 in a smooth closed Riemannian manifold $M$. For every fixed smooth codimension-two submanifold
		$\Sigma$, we prove a little-o improvement $o(\lambda^{1/2}\sqrt{\log\lambda})$ on the classical estimate $O(\lambda^{1/2}\sqrt{\log\lambda})$ of Burq--G\'erard--Tzvetkov \cite{BGT} and Hu \cite{Hu}. Our proof uses the Bargmann transform and Tataru's phase-space representation \cite{TataruPhaseSpace} to reduce the problem to Stein--Street's estimate \cite{SS}  for singular Radon transforms. The three-dimensional case can be handled directly by Ricci--Stein's estimate \cite{RicciStein}.   Moreover, we construct explicit examples to show that the little-o improvement is optimal in general.  These are closely related to earlier counterexamples for  endpoint Strichartz estimates by Montgomery--Smith \cite{MontgomerySmith} and Carbery--Hofmann. In particular, we establish the log-free estimate $O(\lambda^{1/2})$
		when $(M,g)$ and $\Sigma$ are real analytic, whereas this estimate
		fails in general in the smooth setting.
	\end{abstract}
	\maketitle
	
	\renewcommand{\theequation}{\arabic{equation}}
	\section{Introduction}
	
	Let $(M,g)$ be a smooth closed Riemannian manifold of dimension
	$n\geq3$, and let $\Sigma\subset M$ be a smooth embedded closed
	submanifold.  Let $-\Delta_g$ denote the nonnegative
	Laplace--Beltrami operator.  We write $e_\lambda$ for an $L^2$-normalized
	eigenfunction of frequency $\lambda$, so that
	\begin{equation}\label{eigeq}
		-\Delta_g e_\lambda=\lambda^2e_\lambda,\qquad
		\norm{e_\lambda}_{L^2(M)}=1,\qquad \lambda\geq0.
	\end{equation}
	To measure the concentration of an eigenfunction, an important approach is to study the growth of its $L^p$ norms restricted to submanifolds. This type of estimates was studied by Reznikov \cite{rez04} for Maass forms on hyperbolic surfaces.  General $L^p$ restriction estimates on submanifolds for Laplace eigenfunctions were
	established by Burq--G\'erard--Tzvetkov \cite{BGT} and Hu \cite{Hu}.
	Earlier related results include estimates for Fourier integral operators
	with fold singularities by Greenleaf--Seeger \cite{gs1994} and boundary
	trace estimates for the wave equation by Tataru \cite{tataru}.
	Tacy \cite{Tacy} extended restriction estimates to quasimodes,
	while Blair--Park \cite{BP, BP2} and Huang--Wang--Zhang \cite{HuangWangZhang} treated
	Schr\"odinger eigenfunctions with singular potentials.
	Bourgain \cite{bo2009} related geodesic restriction bounds to global
	$L^p$ estimates.  Improvements under curvature assumptions were obtained
	by Chen--Sogge
	\cite{ChenSogge}, Chen \cite{chen}, Xi--Zhang \cite{xz}, Zhang \cite{zhang}, 
	Blair \cite{mb}, Hezari \cite{hez}, and Park \cite{park}.
	On flat tori, arithmetic methods yield stronger restriction bounds,
	see Bourgain--Rudnick \cite{br2012},
	Huang--Zhang \cite{HuangZhang},
	Zhang--Zhu \cite{ZhangZhu}.
	
	Let $\Sigma\subset M$ be a smooth embedded closed
	submanifold of codimension two. The codimension-two endpoint restriction
	estimate of Burq--G\'erard--Tzvetkov \cite{BGT} gives
	\begin{equation}\label{bgtest}
		\norm{e_\lambda}_{L^2(\Sigma)}
		\leq C_{M,g,\Sigma}\lambda^{1/2}\sqrt{\log\lambda}\,
		\norm{e_\lambda}_{L^2(M)},\quad \lambda\geq2.
	\end{equation}
	It
	remains valid for every closed embedded $C^{2,1}$ submanifold
	$\Sigma$ of codimension two by flattening the embedding and applying Blair \cite[Theorem 1.1]{Blair}. On the other hand, the examples given
	in \cite{BGT} on the standard round sphere $S^n$ exhibit growth of
	order $\lambda^{1/2}$, with no logarithmic factor. Thus, a  long-standing problem has been to determine the optimal growth
	rate in the codimension-two endpoint restriction estimate.
	
	In view of the examples on the sphere, it has long been expected that
	the logarithmic loss in \eqref{bgtest} can be removed in general, and that for
	every fixed smooth codimension-two submanifold $\Sigma$ in a general
	smooth manifold $M$, one should have the log-free estimate
	\begin{equation}\label{logfree}
		\norm{e_\lambda}_{L^2(\Sigma)}
		\leq C_{M,g,\Sigma}\lambda^{1/2}
		\norm{e_\lambda}_{L^2(M)}.
	\end{equation}
	See e.g. \cite{BGT}, \cite{ChenSogge}, \cite[Problem~1]{WangZhang}, and
	\cite{HuangWangZhang}. A number of works have established the
	log-free estimate in important special geometric settings.
	Chen--Sogge \cite{ChenSogge} proved \eqref{logfree} for geodesics
	in dimension three. In our earlier work \cite{WangZhang}, we proved
	it for totally geodesic codimension-two submanifolds in all dimensions
	$n\geq3$, as well as for curves with nonvanishing geodesic curvature
	in dimension three. That work also treats real-analytic curves and curves satisfying a
	finite-type condition.
	
	In this paper, we give a complete characterization of the optimal
	growth rate in the codimension-two endpoint restriction estimate.
	We first prove that the Burq--G\'erard--Tzvetkov bound always admits
	a little-o improvement for smooth codimension-two submanifolds, and
	then show that no prescribed quantitative rate of improvement is
	possible in general. Moreover, for real-analytic $(M,g)$ and $\Sigma$, we
	establish the log-free estimate in every dimension.

	\begin{mainthm}\label{thm:general}
		Let $(M,g)$ be a smooth closed Riemannian manifold of dimension
		$n\geq3$, and let $\Sigma\subset M$ be a smooth embedded closed
		submanifold of codimension 2. There are constants
		$B=B(M,g,\Sigma)$ and $A_m=A_m(M,g,\Sigma)$, $m=2,3,\ldots$, such that
		every eigenfunction $e_\lambda$ with $\lambda\geq2$ satisfies
		\begin{equation}\label{eq:general-quant}
			\norm{e_\lambda}_{L^2(\Sigma)}^2
			\leq\left(A_m\lambda+\frac Bm\lambda\log\lambda\right)\norm{e_\lambda}_{L^2(M)}^2.
		\end{equation}
		In particular, as $\lambda\to\infty$,
		\begin{equation}\label{eq:general-littleoh}
			\norm{e_\lambda}_{L^2(\Sigma)}
			=o\!\left(\lambda^{1/2}\sqrt{\log\lambda}\right)\norm{e_\lambda}_{L^2(M)}.
		\end{equation}
		
		If $(M,g)$ and the embedding $\Sigma\hookrightarrow M$
		are real analytic, then the log-free estimate \eqref{logfree} holds.
	\end{mainthm}

	\begin{mainthm}\label{cor:product-lower}
		Let $\Sph^d=\{x\in\R^{d+1}:|x|=1\}$ be the standard round sphere, identify $\Sph^2$ with
		$S^3\cap\{x_4=0\}$, and set
		$\lambda_\nu=\sqrt{\nu^2-1}$. Fix a smooth closed Riemannian
		manifold $Y$ of dimension $n-3$ and equip $M=S^3\times Y$ with
		the product metric $g$.  In each part below, submanifolds are fixed, closed, embedded, contained
		in $\Sph^2\times Y$, and of codimension two in $M$. Sequences
		$(\nu_j,e_{\lambda_{\nu_j}})$ consist of strictly increasing integers $\nu_j\geq3$ and
		real normalized eigenfunctions satisfying \eqref{eigeq}.
		\begin{enumerate}
			\item[(i)] For every integer $K\geq1$, there are a $C^K$ submanifold
			$\Sigma_K$ and a sequence $(\nu_j,e_{\lambda_{\nu_j}})$ such that
			\begin{equation}\label{eq:product-finite}
				c_K\lambda_{\nu_j}^{1/2}\sqrt{\log\lambda_{\nu_j}}
				\leq\|e_{\lambda_{\nu_j}}\|_{L^2(\Sigma_K)}
				\leq C_K\lambda_{\nu_j}^{1/2}\sqrt{\log\lambda_{\nu_j}}.
			\end{equation}
			Here $0<c_K\leq C_K$ depend only on $K$. 
			
			\item[(ii)] For every $L:[2,\infty)\to(0,\infty)$ with
			$L(\lambda)=o(\sqrt{\log\lambda})$, there are a smooth submanifold
			$\Sigma$ and a sequence $(\nu_j,e_{\lambda_{\nu_j}})$ such that
			\begin{equation}\label{eq:product-no-rate}
				\frac{\|e_{\lambda_{\nu_j}}\|_{L^2(\Sigma)}}
				{\lambda_{\nu_j}^{1/2}L(\lambda_{\nu_j})}\geq j.
			\end{equation}
		\end{enumerate}
	\end{mainthm}
	
	Thus, no prescribed replacement
	$L=o(\sqrt{\log\lambda})$ gives a bound valid for all smooth $\Sigma$. 
	In particular,  Theorem~\ref{cor:product-lower}(ii) gives one smooth $\Sigma$ and one
	sequence $(\nu_j,e_{\lambda_{\nu_j}})$ such that
	\[
	\frac{\|e_{\lambda_{\nu_j}}\|_{L^2(\Sigma)}}
	{\lambda_{\nu_j}^{1/2}(\log\lambda_{\nu_j})^{\alpha}}
	\longrightarrow\infty
	\qquad\text{for every fixed }\alpha<\frac12.
	\]

	The logarithmic behavior at the codimension-two endpoint has a close
	parallel in endpoint Strichartz estimates. For the two-dimensional
	Schr\"odinger equation, the endpoint Strichartz estimate fails, while
	a localized version holds with a logarithmic loss. After a $TT^*$
	argument, both problems exhibit the borderline singularity
	$(t-s)^{-1}$. Moreover, the phase appearing in the
	counterexample of Montgomery--Smith \cite{MontgomerySmith} has a close
	analogue in our lower-bound construction. See also Tao \cite{Tao}.
	
	A related regularity phenomenon occurs for oscillatory singular
	integrals. Stein--Wainger \cite{SteinWainger} obtained uniform bounds
	for an oscillatory singular integral with polynomial phases, while Zhang--Zhu \cite{ZhangZhuSW} showed that
	the corresponding integral is
	$O(\log\lambda)$ for finite $C^K$ phases but $o(\log\lambda)$ for each
	smooth phase. Theorems~\ref{thm:general} and~\ref{cor:product-lower}
	show an analogous distinction for codimension-two restriction. See also Nagel--Wainger \cite{nw}, Nagel--Vance--Wainger--Weinberg \cite{nvww}, Phong--Stein \cite{ps86}, Ricci--Stein \cite{RicciStein}, Pan \cite{pan91,pan93,Pan}, Seeger \cite{seeger94}, Carbery--P\'erez \cite{cp99}.

	\subsection*{Proof outline}
Section~\ref{sec:general-n} proves Theorem~\ref{thm:general} by a
localized $TT^*$ argument.
Stationary phase in the $n-3$ transverse
variables reduces the problem to oscillatory Hilbert-type operators,
up to an $O(\lambda)$ error in $L^2$ operator norm. We split the
resulting kernel at $|t-s|=\lambda^{-1/m}$. The far part is bounded by
$Cm^{-1}\log\lambda$ using uniform transverse $L^2$ estimates and
Schur's test. For the near part, we use the Bargmann transform and Tataru's phase-space representation \cite{TataruPhaseSpace} to  reduce the problem to the
singular Radon transform estimate of Stein--Street
\cite[Corollary~5.6]{SS}. This yields a bound $C_m$ for the
near part.  Moreover, we apply the same
theorem directly to the real-analytic maps to obtain \eqref{logfree}. Appendix~\ref{sec:general} gives a direct proof for the smooth case in dimension
three using the estimate of Ricci--Stein \cite{RicciStein}. 

Section~\ref{sec:finite-proof} proves Theorem~\ref{cor:product-lower} by constructing curves on the round sphere $S^3$. The construction is inspired by the counterexamples in Montgomery--Smith \cite{MontgomerySmith}.
The exact spherical
projector and lacunary trigonometric profiles produce logarithmic
lower bounds on short graph arcs.   For part~(i), rescaling and gluing
these arcs yields a $C^K$ closed embedded curve attaining the full
$\sqrt{\log\lambda}$ loss, while Schur's test gives the matching upper
bound. For part~(ii), the scales and frequencies are chosen according
to $L$ to produce a smooth closed embedded curve whose restriction
norms exceed $\lambda^{1/2}L(\lambda)$ by arbitrarily large factors
along a sequence of frequencies. Taking products with $Y$ extends the
constructions to all dimensions $n\geq3$.

	\subsection*{Notation}
	Constants $C,c>0$ may change from line to line. Subscripts record allowed
	additional dependence. 
	We write $O_{C^N}(a)$ for a function whose
	$C^N$ norm is at most $C_Na$. For scalar functions, subscripts denote partial derivatives:
	$\Phi_{x\eta}=(\partial_{x_i}\partial_{\eta_j}\Phi)_{ij}$
	is the mixed Hessian. For a differentiable map, $\partial_\xi$ denotes its Jacobian
	with respect to $\xi$.
	
	\subsection*{Acknowledgments.}
	The authors would like to thank Christopher Sogge for  earlier discussions.  C.Z. is partially supported by National Key R\&D Program of China No. 2024YFA1015300 and NSFC Grant No. 12371097. X.W. is partially supported by NSFC Grant No. 12671119. 
	
	\renewcommand{\theequation}{\thesection.\arabic{equation}}
	
	\section{Proof of Theorem 1}
	\label{sec:general-n}
	
	Write $g_\Sigma=g|_{T\Sigma}$ for its induced metric,
	$\dd V_g$ and $\dd\sigma$ for the Riemannian measures on $M$ and $\Sigma$.
	For $d\in\N_0$ and an
	operator-valued kernel $\mathcal K(t,s):L^2(\R^d)\to L^2(\R^d)$, define
	\[
	(\mathcal S f)(t)=\int_\R\mathcal K(t,s)f(s)\dd s.
	\]
	Then we use Schur's test in the following form
	\[
		\norm{\mathcal S}_{L^2(\R^{d+1})\to L^2(\R^{d+1})}
		\leq\Big(\sup_t\int_\R
		\norm{\mathcal K(t,s)}_{L^2(\R^d)\to L^2(\R^d)}\dd s\Big)^{\frac12}
	\Big(\sup_s\int_\R
		\norm{\mathcal K(t,s)}_{L^2(\R^d)\to L^2(\R^d)}\dd t\Big)^{\frac12}.
	\]
	The scalar case is $d=0$, with the inner norms replaced by absolute values.
	
	\subsection{The \texorpdfstring{$TT^*$}{TT*} reduction}
	\label{sec:upper-tt}
	
	Let $d_g$ be the Riemannian distance, $P=\sqrt{-\Delta_g}$, and
	$\RR f=f|_\Sigma$. Choose real even $\varrho\in\mathcal S(\R)$ with
	$\varrho(0)=1$ and sufficiently small $\supp\widehat\varrho$.
	Let $\chi=\varrho^2$ and
	$T_\lambda=\RR\varrho(\lambda-P):L^2(M)\to L^2(\Sigma)$.
	Consequently,
	\[
	T_\lambda e_\lambda=e_\lambda,\qquad
	T_\lambda T_\lambda^*=\RR\chi(\lambda-P)\RR^*.
	\]
	Fix a finite family of real coordinate cutoffs $\beta$ with
	$\sum_\beta\beta^2=1$ on $\Sigma$.  Let $\mathcal K_\lambda$ be the
	coordinate representation of $\beta T_\lambda T_\lambda^*\beta$ on $L^2$, with the density factors absorbed into its kernel
	$K_\lambda(x,y)$.  Then
	\[
	\norm{\mathcal K_\lambda}_{L^2(\R^{n-2})\to L^2(\R^{n-2})}
	=\norm{\beta T_\lambda}_{L^2(M)\to L^2(\Sigma)}^2,
	\qquad
	\norm{T_\lambda f}_{L^2(\Sigma)}^2
	=\sum_\beta\norm{\beta T_\lambda f}_{L^2(\Sigma)}^2.
	\]
	
	\subsection{The spectral kernel and transverse stationary phase}
	\label{sec:upper-kernel}
	
	Fix a sufficiently small $r_0\in(0,1)$.  For fixed
	$b\in C_c^\infty(\R^{2n-4})$ and a real smooth phase $\Phi$, define
	\begin{equation}\label{upper:Vdef}
		[V_\lambda(t,s)h](z)
		=\left(\frac\lambda{2\pi}\right)^{n-3}
		\iint_{\R^{n-3}\times\R^{n-3}}
		e^{i\lambda(\Phi(t,s,z,\eta)-w\cdot\eta)}
		b(t,s,z,\eta)h(w)\dd w\dd\eta,
	\end{equation}
	where $t,s\in\R$, $z,w,\eta\in\R^{n-3}$, and, for $|t-s|\leq r_0$,
	\begin{equation}\label{upper:phase}
		\Phi(t,s,z,\eta)=z\cdot\eta+(t-s)\phi(t,s,z,\eta),
		\qquad
		\norm{\Phi_{z\eta}-I}_{\R^{n-3}\to\R^{n-3}}<\tfrac12.
	\end{equation}
	Fix even $\chi_0,\chi_1\in C_c^\infty(\R)$ with
	$0\leq\chi_0,\chi_1\leq1$, $\chi_0=1$ near zero,
	$\supp\chi_0\subset[-r_0,r_0]$, $\chi_1=1$ on $[-1,1]$, and
	$\supp\chi_1\subset[-2,2]$.  For a fixed sufficiently large $C_0\geq2$,
	depending only on the geometric data and cutoffs, set
	\[
	k_\lambda(r)=\frac{\chi_0(r)(1-\chi_1(\lambda r/C_0))}{r},
	\qquad r\ne0,\quad k_\lambda(0)=0,\quad\lambda\geq1.
	\]
	The cutoffs and $C_0$ are independent of $m,\lambda$.
	The associated operator on $L^2(\R^{n-2})$ is
	\begin{equation}\label{upper:Hdef}
		[H_\lambda f](t,z)
		=\int_\R k_\lambda(t-s)
		[V_\lambda(t,s)f(s,\cdot)](z)\dd s.
	\end{equation}
	When $n=3$, the transverse variables and integrals are omitted, and
	empty determinants are interpreted as $1$.

\begin{proposition}[Transverse factorization]\label{upper:transverse-factorization}
		After shrinking the coordinate patches, there are finitely many pairs
		$V_{\lambda,j}^\pm$, $1\leq j\leq J$, of the form
		\eqref{upper:Vdef}--\eqref{upper:phase}, with associated operators
		$H_{\lambda,j}^\pm$ defined by \eqref{upper:Hdef}, such that
		\begin{equation}\label{upper:decomposition}
			\mathcal K_\lambda
			=\sum_{j=1}^J\frac\lambda i
			(H_{\lambda,j}^+-H_{\lambda,j}^-)+\mathcal E_\lambda,
			\qquad
			\norm{\mathcal E_\lambda}_{L^2(\R^{n-2})\to L^2(\R^{n-2})}
			\leq C\lambda.
		\end{equation}
		Each summand is expressed in fixed orthogonal coordinates and then
		transported back to the original patch.  The coordinate changes, phases,
		amplitudes, $J$, and $C$ are independent of $m$ and $\lambda\geq2$.
	\end{proposition}
	
	The proof uses the next three lemmas.  In a coordinate patch
	$\Omega\subset\R^{n-2}$ containing $\supp\beta$, the localized kernel is
	\[
	K_\lambda(x,y)=\beta(x)\beta(y)
	\bigl(\det g_\Sigma(x)\det g_\Sigma(y)\bigr)^{1/4}
	[\chi(\lambda-P)](x,y),\qquad x,y\in\Omega,
	\]
	where the bracketed kernel is taken with respect to $\dd V_g$.
	Choose $\Omega$ and $\supp\widehat\chi$ sufficiently small for the
	short-time Hadamard parametrix, with $d_g^2$ smooth and $d_g<1$ on
	$\Omega\times\Omega$.

\begin{lemma}[Kernel expansion]\label{upper:local-spectral}
		For every $\lambda\geq2$,
		$\norm{K_\lambda}_{L^\infty(\Omega\times\Omega)}\leq C\lambda^{n-1}$.
		There is a fixed $a_0\in C_c^\infty(\Omega\times\Omega)$ such that,
		when $\lambda d_g(x,y)\geq2$,
		\begin{equation}\label{upper:spectral-remainder}
			\begin{gathered}
				K_\lambda(x,y)=K_{\lambda,0}(x,y)+K_{\lambda,1}(x,y)+S_\lambda(x,y),\\
				K_{\lambda,0}(x,y)
				=\lambda^{\frac{n-1}{2}}d_g(x,y)^{-\frac{n-1}{2}}a_0(x,y)
			\sum_\pm e^{\pm i\lambda d_g(x,y)\mp i\pi(n-1)/4},\\
				K_{\lambda,1}(x,y)=
				\lambda^{\frac{n-1}{2}}d_g(x,y)^{-\frac{n-1}{2}}
				\sum_\pm e^{\pm i\lambda d_g(x,y)}
				a_\lambda^\pm\bigl(x,y,\lambda d_g(x,y)\bigr).
			\end{gathered}
		\end{equation}
		Here 
		$S_\lambda$ is smooth, has fixed compact support in $\Omega\times\Omega$,
		and satisfies $\norm{S_\lambda}_{L^\infty(\Omega\times\Omega)}\leq C$.
		The smooth amplitudes $a_\lambda^\pm(x,y,\rho)$ have fixed compact $(x,y)$
		support and satisfy,
		\[
		\left|\partial_{x,y}^{\kappa}\partial_{\rho}^j
		a_\lambda^\pm(x,y,\rho)\right|
		\leq C_{\kappa,j}\rho^{-1-j},\qquad
		2\leq \rho\leq\lambda,\quad
		\kappa\in\mathbb N_0^{2n-4},\quad j\in\mathbb N_0.
		\]
		All constants are independent of $m,\lambda$.
	\end{lemma}

	\begin{proof}
		Fourier inversion gives
		\[
		\chi(\lambda-P)
		=\frac1\pi\int_\R e^{i\lambda t}\widehat\chi(t)\cos(tP)\dd t
		-\chi(\lambda+P).
		\]
		Fix the  truncation order sufficiently large that the
		integrated remainder is bounded, independently of $m,\lambda$.
		By \cite{hangzhou} or \cite[Lemma~1, (2.3)--(2.4)]{WangZhang} and polar coordinates,
		$K_\lambda-S_\lambda$ is a finite sum of terms
		\[
		b_*(x,y)\int_{\lambda/2}^{2\lambda}\chi_*(\lambda-\tau)\tau^{n-1-l}
		\int_{\Sph^{n-1}}e^{i\tau d_g(x,y)\omega_1}\dd\omega\dd\tau,
		\]
		where $b_*\in C_c^\infty(\Omega\times\Omega)$ and
		$\chi_*\in\mathcal S(\R)$ are fixed, and $l\in\mathbb N_0$.
		The unique term with $l=0$ has $(b_*,\chi_*)=(b_0,\chi)$.
		The discarded radial tails and $\chi(\lambda+P)$ are rapidly decreasing
		in $\lambda$, and the integrated Hadamard remainder is bounded.
		Thus $\norm{S_\lambda}_{L^\infty(\Omega\times\Omega)}\leq C$.
		Each displayed radial integral is smooth in $(x,y)$, since the spherical
		integral is a smooth function of $d_g(x,y)^2$. Hence $S_\lambda$ is smooth.
		Moreover,
		\[
		\begin{aligned}
			\int_{\lambda/2}^{2\lambda}|\chi_*(\lambda-\tau)|\tau^{n-1-l}
			\left|\int_{\Sph^{n-1}}e^{i\tau d_g(x,y)\omega_1}\dd\omega\right|\dd\tau
			\leq C\lambda^{n-1-l}\norm{\chi_*}_{L^1(\R)}\leq C\lambda^{n-1}.
		\end{aligned}
		\]
		This proves $\norm{K_\lambda}_{L^\infty(\Omega\times\Omega)}\leq C\lambda^{n-1}$.
		
		Stationary phase \cite[Section~1.1, (1.1.20), and Theorem~1.2.1]{SoggeFIO} gives
		\[
		\begin{gathered}
			\int_{\Sph^{n-1}}e^{i\rho\omega_1}\dd\omega
			=(2\pi)^{\frac{n-1}{2}}\rho^{-\frac{n-1}{2}}
			\sum_\pm e^{\pm i\rho\mp i\pi(n-1)/4}(1+c_\pm(\rho)),\\
			|c_\pm^{(j)}(\rho)|\leq C_j \rho^{-1-j},\qquad \rho\geq1,\quad j\geq0.
		\end{gathered}
		\]
		For each radial term, set
		\[
		A_{\lambda,l}^\pm(\rho)
		=\lambda^{-l}\int_{\lambda/2}^{2\lambda}\chi_*(\lambda-\tau)
		(\tau/\lambda)^{\frac{n-1}{2}-l}
		e^{\pm i(\tau-\lambda)\rho/\lambda}(1+c_\pm(\tau \rho/\lambda))\dd\tau.
		\]
		When $\lambda d_g(x,y)\geq2$, its contribution is
		\[
		(2\pi)^{\frac{n-1}{2}}\lambda^{\frac{n-1}{2}}
		d_g(x,y)^{-\frac{n-1}{2}}b_*(x,y)
		\sum_\pm e^{\pm i\lambda d_g(x,y)\mp i\pi(n-1)/4}
		A_{\lambda,l}^\pm(\lambda d_g(x,y)).
		\]
		For $2\leq \rho\leq\lambda$ and $\lambda/2\leq\tau\leq2\lambda$, we have
		\[
		\begin{aligned}
			\left|(\tau/\lambda)^{\frac{n-1}{2}-l}-1\right|
			&\leq C|\tau-\lambda|/\lambda,\\
			|\partial_{\rho}^k c_\pm(\tau \rho/\lambda)|
			&=(\tau/\lambda)^k|c_\pm^{(k)}(\tau \rho/\lambda)|
			\leq C_k \rho^{-1-k},\\
			\int_\R(1+|\tau-\lambda|)^{j+1}|\chi_*(\lambda-\tau)|\dd\tau
			&\leq C_j.
		\end{aligned}
		\]
		Since
		\[
		\int_\R\chi_*(\lambda-\tau)e^{\pm i(\tau-\lambda)\rho/\lambda}\dd\tau
		=\widehat{\chi_*}(\pm \rho/\lambda),
		\]
		the product rule and the Schwartz tails give
		\[
		\begin{aligned}
			\left|\partial_{\rho}^j\left[
			A_{\lambda,l}^\pm(\rho)-\lambda^{-l}\widehat{\chi_*}(\pm \rho/\lambda)
			\right]\right|\leq C_j\lambda^{-l}
			\left(\lambda^{-j-1}
			+\sum_{k=0}^j\lambda^{-(j-k)}\rho^{-1-k}\right)
			\leq C_j \rho^{-1-j}.
		\end{aligned}
		\]
		Here the first term bounds the error from the power factor and the tails,
		and the sum bounds the derivatives of $c_\pm$. For $l\geq1$,
		\[
		\left|\partial_{\rho}^j\left[
		\lambda^{-l}\widehat{\chi_*}(\pm \rho/\lambda)\right]\right|
		\leq C_j\lambda^{-l-j}\leq C_j \rho^{-1-j},
		\]
		so $|\partial_{\rho}^jA_{\lambda,l}^\pm(\rho)|\leq C_j \rho^{-1-j}$.
		For $l=0$, evenness of $\widehat\chi$ gives the principal amplitude
		\[
		a_0(x,y)=(2\pi)^{\frac{n-1}{2}}b_0(x,y)\widehat\chi(d_g(x,y)).
		\]
		It is smooth across the diagonal since $\widehat\chi$ is even and
		$d_g^2$ is smooth. Collect the $l=0$ differences
		$A_{\lambda,0}^\pm(\rho)-\widehat\chi(\pm \rho/\lambda)$ and the $l\geq1$
		terms, multiplied by $(2\pi)^{(n-1)/2}b_*(x,y)e^{\mp i\pi(n-1)/4}$,
		into $a_\lambda^\pm(x,y,\rho)$. At fixed $\rho$, their $(x,y)$ derivatives
		fall only on $b_*$, so
		\[
		|\partial_{x,y}^\kappa\partial_{\rho}^j a_\lambda^\pm(x,y,\rho)|
		\leq C_{\kappa,j}\rho^{-1-j},\qquad 2\leq \rho\leq\lambda.
		\]
		Substituting $\rho=\lambda d_g(x,y)$ proves \eqref{upper:spectral-remainder}.
	\end{proof}
	Fix a finite smooth partition of unity in the direction
	$(x-y)/|x-y|$, independent of $m,\lambda$, with each cutoff
	invariant under $x-y\mapsto y-x$ and supported in a pair of
	sufficiently small opposite cones. Fix an arbitrary piece and
	suppress its index. In suitable fixed orthogonal coordinates, write
	\[
	x=(t,z),\qquad y=(s,w),\qquad
	v=\frac{z-w}{t-s}\quad(t\ne s).
	\]
	The cutoff then has the form
	$\wt\psi(v)\in C_c^\infty(\{|v|<c_1\})$, with $c_1>0$ fixed.
	After shrinking the coordinate patch, the distance quotient
	defined below has a uniformly positive definite $v$-Hessian
	for $|v|\leq c_1$. The following argument applies to each piece
	in its corresponding coordinates, and the resulting operators
	are summed after returning to the original coordinates.
	For $n=3$, take $\wt\psi=1$.
	For $|v|\leq c_1$ and $|t-s|\leq r_0$, the distance quotient extends
	smoothly to $t=s$.  After shrinking the coordinate patch, it satisfies
	\begin{equation}\label{upper:distance}
		\begin{gathered}
			d_g\bigl((t,z),(s,z-(t-s)v)\bigr)
			=|t-s|\ell(t,s,z,v),\\
			\ell(t,t,z,v)=|(1,v)|_{g_\Sigma(t,z)},\qquad
			0<c\leq\ell\leq C,\qquad \ell_{vv}\geq cI.
		\end{gathered}
	\end{equation}
	Indeed, Taylor's formula for $d_g^2$ gives the smooth extension and
	its diagonal value.  At $t=s$, for
	$e\in\R^{n-3}\setminus\{0\}$,
	\[
	e^T\ell_{vv}e
	=\frac{|(0,e)|_{g_\Sigma}^2|(1,v)|_{g_\Sigma}^2
		-\langle(0,e),(1,v)\rangle_{g_\Sigma}^2}
	{|(1,v)|_{g_\Sigma}^3}>0.
	\]
	Compactness and a smaller coordinate patch give
	$0<c\leq\ell\leq C$ and $\ell_{vv}\geq cI$ on a fixed neighborhood
	of $|v|\leq c_1$.

\begin{lemma}[Transverse kernel]\label{upper:kernel-construction}
After a fixed refinement of the directional partition, each piece
admits operators $V_\lambda^\pm(t,s)$ of the form
\eqref{upper:Vdef}--\eqref{upper:phase}, with phases and amplitudes
independent of $m,\lambda$, whose kernels satisfy the following
identities and estimates for $\lambda\geq2$ and
$C_0/\lambda\leq|t-s|\leq r_0$.
	\begin{equation}\label{upper:transversekernels}
		\begin{gathered}
			V_\lambda^\pm=V_{\lambda,0}^\pm+E_\lambda^\pm,\\
			V_{\lambda,0}^\pm(t,s;z,w)
			=\lambda^{\frac{n-3}{2}}|t-s|^{-\frac{n-3}{2}}
			\ell^{-\frac{n-1}{2}}a_0(x,y)\wt\psi(v)
			e^{\pm i\lambda\sgn(t-s)d_g(x,y)
				\mp i\pi(n-3)\sgn(t-s)/4},
		\end{gathered}
	\end{equation}
	where $x=(t,z)$, $y=(s,w)$, $v=(z-w)/(t-s)$, and
	$\ell=\ell(t,s,z,v)$.
	Moreover,
	\[
	E_\lambda^\pm(t,s;z,w)
	=\lambda^{\frac{n-3}{2}}|t-s|^{-\frac{n-3}{2}}
	e^{\pm i\lambda\sgn(t-s)d_g(x,y)}
	a_{\lambda,1}^\pm(t,s,z,v)+F_\lambda^\pm(t,s;z,w),
	\]
	where $a_{\lambda,1}^\pm$ has fixed compact support where
	\eqref{upper:distance} holds, and, for every integer $k\geq0$,
	\[
	\begin{aligned}
		\norm{a_{\lambda,1}^\pm(t,s,\cdot,\cdot)}_{C^k(\R^{2n-6})}
		&\leq C_k(\lambda|t-s|)^{-1},\\
		|F_\lambda^\pm(t,s;z,w)|
		&\leq C_k\lambda^{n-3}
		(1+\lambda|t-s|+\lambda|z-w|)^{-k}.
	\end{aligned}
	\]
	For $n=3$, $E_\lambda^\pm=0$.
	The principal kernels satisfy
	\begin{equation}\label{pid}
		\frac{\lambda}{i(t-s)}
		\bigl(V_{\lambda,0}^+-V_{\lambda,0}^-\bigr)
		=\wt\psi(v)K_{\lambda,0}(x,y).
	\end{equation}
	All constants are independent of $m,\lambda$.
\end{lemma}

\begin{proof}
	For $n\geq4$, the bound $\ell_{vv}\geq cI$ gives smooth local
	inverses $v_\pm=v_\pm(t,s,z,\eta)\in\R^{n-3}$ determined by $	\eta=\pm\ell_v(t,s,z,v_\pm).$
	Define
	\[
	\begin{aligned}
		\Phi_\pm(t,s,z,\eta)
		&=z\cdot\eta+(t-s)(\pm\ell-v_\pm\cdot\eta),\\
		b^\pm(t,s,z,\eta)
		&=(2\pi)^{\frac{n-3}{2}}
	\sqrt{\det\ell_{vv}^{-1}}\ell^{-\frac{n-1}{2}}
			a_0((t,z),(s,z-(t-s)v_\pm))\wt\psi(v_\pm),
	\end{aligned}
	\]
	with $\ell,\ell_{vv}$ evaluated at $(t,s,z,v_\pm)$.
	Choose a fixed $\psi\in C_c^\infty(\R^{n-3})$, equal to one near
	$\supp\wt\psi$ and supported where this construction is valid.
	Extend $b^\pm$ by zero and the phase corrections smoothly with
	compact support, preserving the formulas for $v_\pm\in\supp\psi$
	on the coordinate support.
	
	Differentiation gives
	\[
	\begin{gathered}
		\partial_\eta v_\pm=\pm\ell_{vv}^{-1},\qquad
		\partial_\eta\Phi_\pm-w=(t-s)(v-v_\pm),\\
		\Phi_{\pm,\eta\eta}=\mp(t-s)\ell_{vv}^{-1},\qquad
		\norm{\Phi_{\pm,z\eta}-I}_{\R^{n-3}\to\R^{n-3}}
		\leq C|t-s|.
	\end{gathered}
	\]
	Thus \eqref{upper:phase} holds after decreasing $r_0$.
	At the critical point $\eta=\pm\ell_v(t,s,z,v)$,
	\[
	\Phi_\pm-w\cdot\eta
	=\pm(t-s)\ell
	=\pm\sgn(t-s)d_g(x,y),
	\]
	and the Hessian $\Phi_{\pm,\eta\eta}$ has signature
	$\mp(n-3)\sgn(t-s)$.
	The leading coefficient is
	\[
	\begin{aligned}
		&\left(\frac{\lambda}{2\pi}\right)^{n-3}
		\left(\frac{2\pi}{\lambda|t-s|}\right)^{\frac{n-3}{2}}
		b^\pm(t,s,z,\pm\ell_v(t,s,z,v))
		\sqrt{\det\ell_{vv}(t,s,z,v)}\\
		&\qquad
		=\lambda^{\frac{n-3}{2}}|t-s|^{-\frac{n-3}{2}}
		\ell^{-\frac{n-1}{2}}a_0(x,y)\wt\psi(v).
	\end{aligned}
	\]
	Stationary phase
	\cite[Corollary~1.1.8 and (1.1.20)]{SoggeFIO}
	with parameter $\lambda|t-s|$ therefore gives
	\[
	\psi(v)V_\lambda^\pm
	=V_{\lambda,0}^\pm
	+\lambda^{\frac{n-3}{2}}|t-s|^{-\frac{n-3}{2}}
	e^{\pm i\lambda\sgn(t-s)d_g(x,y)}
	a_{\lambda,1}^\pm(t,s,z,v),
	\]
	with the stated support and derivative bounds.
	
	Set $F_\lambda^\pm=(1-\psi(v))V_\lambda^\pm$.
	On $\supp b^\pm$, $v_\pm\in\supp\wt\psi$.
	Hence, where $1-\psi(v)\ne0$,
	\[
	\begin{aligned}
		|\partial_\eta\Phi_\pm-w|
		&=|t-s|\,|v-v_\pm|
		\geq c(|t-s|+|z-w|),\\
		|\partial_\eta^\alpha(\partial_\eta\Phi_\pm-w)|
		&\leq C_\alpha|t-s|,
		\qquad \alpha\in\mathbb N_0^{n-3},\quad |\alpha|\geq1.
	\end{aligned}
	\]
	Repeated integration by parts in $\eta$ gives the stated decay
	bound for $F_\lambda^\pm$.
	
	For $n=3$, take
	\[
	\Phi_\pm(t,s)=\pm(t-s)\ell(t,s),\qquad
	b^\pm(t,s)=\ell(t,s)^{-1}a_0(t,s).
	\]
	Then $V_\lambda^\pm=V_{\lambda,0}^\pm$.
	
	Finally, for every $n\geq3$, using $d_g(x,y)=|t-s|\ell$ gives
	\[
	\begin{aligned}
		\frac{\lambda}{i(t-s)}
		\bigl(V_{\lambda,0}^+-V_{\lambda,0}^-\bigr)
		&=2\lambda^{\frac{n-1}{2}}d_g(x,y)^{-\frac{n-1}{2}}
		a_0(x,y)\wt\psi(v)
		\sin\left(\lambda d_g(x,y)-\frac{\pi(n-3)}4\right)\\
		&=\wt\psi(v)K_{\lambda,0}(x,y),
	\end{aligned}
	\]
	which proves \eqref{pid}.
\end{proof}

\begin{lemma}[Transverse $L^2$ estimates]\label{upper:transverse-L2}
	Let $a(t,s,z,v)$ be smooth with fixed compact support where
	\eqref{upper:distance} holds. Define
		\begin{equation}\label{upper:transverse-operator}
			[\mathcal T_{\lambda,a}^\pm(t,s)h](z)
			=\lambda^{\frac{n-3}{2}}|t-s|^{-\frac{n-3}{2}}
			\int_{\R^{n-3}}e^{\pm i\lambda d_g((t,z),(s,w))}
			a\!\left(t,s,z,\frac{z-w}{t-s}\right)h(w)\dd w.
		\end{equation}
		For a fixed sufficiently large integer $k$,
		\begin{equation}\label{upper:transverse-estimate}
			\norm{\mathcal T_{\lambda,a}^\pm(t,s)}_{L^2(\R^{n-3})\to L^2(\R^{n-3})}
			\leq C\norm{a(t,s,\cdot,\cdot)}_{C^k(\R^{2n-6})},
			\qquad \lambda|t-s|\geq1.
		\end{equation}
		Let $\mathcal E_\lambda^\pm(t,s)$ be the integral operator with kernel
		$E_\lambda^\pm(t,s;z,w)$ from Lemma~\ref{upper:kernel-construction}.  Then
		\[
		\norm{\mathcal E_\lambda^\pm(t,s)}_{L^2(\R^{n-3})\to L^2(\R^{n-3})}
		\leq C(\lambda|t-s|)^{-1},\qquad \lambda|t-s|\geq C_0.
		\]
		Both estimates are uniform for $0<|t-s|\leq r_0$.
	\end{lemma}
	
	\begin{proof}
		For $n=3$, both estimates are immediate.  Suppose $n\geq4$ and fix
		$0<|t-s|\leq r_0$.  Partition $\R^{n-3}$ into disjoint cubes $Q$ of
		side $|t-s|$.  Insert smooth output cutoffs supported in $2Q$ and equal
		to one on $Q$.  Since $|z-w|\leq C|t-s|$ on the kernel support, only
		the values of $h$ on a fixed concentric enlargement $Q^*$ contribute,
		and $\sum_Q\ind_{Q^*}\leq C$.
		
		Rescale about the center $z_Q$:
		\[
		z=z_Q+|t-s|z',\qquad w=z_Q+|t-s|w',\qquad
		v=\sgn(t-s)(z'-w').
		\]
		The $L^2$-normalized rescaling gives the prefactor and mixed Hessian
		\[
		\begin{gathered}
			\lambda^{\frac{n-3}{2}}|t-s|^{-\frac{n-3}{2}}|t-s|^{n-3}
			= (\lambda|t-s|)^{\frac{n-3}{2}},\\
			\partial_{z'}\partial_{w'}\!\left(\frac{d_g(x,y)}{|t-s|}\right)
			=-\ell_{vv}-(t-s)\ell_{zv}.
		\end{gathered}
		\]
		By \eqref{upper:distance}, this Hessian is uniformly invertible for small
		$r_0$.  On the fixed rescaled support, the phase derivatives are uniformly
		bounded, and the amplitude has $C^k$ norm at most
		$C_k\norm{a(t,s,\cdot,\cdot)}_{C^k(\R^{2n-6})}$.
		Apply H\"ormander's $L^2$ oscillatory integral theorem
		\cite{Hormander1973, SoggeFIO} with parameter $\lambda|t-s|\geq1$
		to obtain the factor $(\lambda|t-s|)^{-(n-3)/2}$.  Hence
		\[
		\norm{\mathcal T_{\lambda,a}^\pm(t,s)h}_{L^2(Q)}
		\leq C\norm{a(t,s,\cdot,\cdot)}_{C^k(\R^{2n-6})}
		\norm h_{L^2(Q^*)}.
		\]
		Squaring and summing proves \eqref{upper:transverse-estimate}, since
		$\sum_Q\norm h_{L^2(Q^*)}^2\leq C\norm h_{L^2(\R^{n-3})}^2$.
		
		By Lemma~\ref{upper:kernel-construction} and
		\eqref{upper:transverse-estimate}, the part of $E_\lambda^\pm$ with
		amplitude $a_{\lambda,1}^\pm$ has operator norm at most
		$C(\lambda|t-s|)^{-1}$, for either phase sign.
		For $F_\lambda^\pm$, its decay bound with exponent $n-2$ gives
		\[
		\begin{aligned}
			&\sup_z\int_{\R^{n-3}}|F_\lambda^\pm(t,s;z,w)|\dd w
			+\sup_w\int_{\R^{n-3}}|F_\lambda^\pm(t,s;z,w)|\dd z\\
			&\qquad\leq C\int_{\R^{n-3}}
			\frac{\lambda^{n-3}\dd w}
			{(1+\lambda|t-s|+\lambda|w|)^{n-2}}
			=\frac{C}{1+\lambda|t-s|}
			\leq\frac{C}{\lambda|t-s|}.
		\end{aligned}
		\]
		Schur's test gives the asserted bound for $\mathcal E_\lambda^\pm(t,s)$.
	\end{proof}
	
	\begin{proof}[Proof of Proposition~\ref{upper:transverse-factorization}]
		Increase the fixed constant $C_0$ so that $C_0\ell\geq2$
		on the relevant supports.
		Use the operators from Lemma~\ref{upper:kernel-construction}, shrinking
		the coordinate patches so that $\chi_0(t-s)=1$ on the $t,s$ supports of
		$K_\lambda$ and $b^\pm$.  On each directional piece, the principal-kernel
		identity gives, with kernel arguments suppressed,
		\[
		\begin{aligned}
			\wt\psi(v)K_\lambda-\frac\lambda i k_\lambda(t-s)
			(V_\lambda^+-V_\lambda^-)&=\chi_1\!\left(\frac{\lambda(t-s)}{C_0}\right)\wt\psi(v)K_\lambda
			+\left[1-\chi_1\!\left(\frac{\lambda(t-s)}{C_0}\right)\right]
			\wt\psi(v)S_\lambda\\
			&+\left[1-\chi_1\!\left(\frac{\lambda(t-s)}{C_0}\right)\right]
			\wt\psi(v)K_{\lambda,1}
			-\frac\lambda i k_\lambda(t-s)(E_\lambda^+-E_\lambda^-).
		\end{aligned}
		\]
		Let $\mathcal N_\lambda$ be the integral operator with the sum of the
		first two kernels on the right.  The first is supported where
		$|x-y|\leq CC_0/\lambda$. The second is bounded on a fixed compact set.
		Lemma~\ref{upper:local-spectral} and Schur's test give
		\[
		\norm{\mathcal N_\lambda}_{L^2(\R^{n-2})\to L^2(\R^{n-2})}
		\leq C\bigl(\lambda^{n-1}(C_0/\lambda)^{n-2}+1\bigr)\leq C\lambda.
		\]
		For fixed $t,s$, let $\mathcal D_\lambda(t,s)$ be the transverse operator
		with the last two kernels. It vanishes unless
		$C_0/\lambda\leq|t-s|\leq r_0$.
		Each positive-order $(z,v)$ derivative of $\lambda|t-s|\ell$
		is $O(\lambda|t-s|)$. Thus Lemma~\ref{upper:local-spectral}
		and the chain rule give, for every integer $k\geq0$,
		\[
		\begin{aligned}
			\norm{
			\wt\psi(v)\ell^{-\frac{n-1}{2}}
			a_\lambda^\pm\bigl((t,z),(s,z-(t-s)v),\lambda|t-s|\ell\bigr)
			}_{C^k(\R^{2n-6})}\leq C_k(\lambda|t-s|)^{-1},
		\end{aligned}
		\]
		where $\ell=\ell(t,s,z,v)$ and the amplitude is extended by zero
		outside its fixed compact support.
		Lemma~\ref{upper:transverse-L2}, applied to this amplitude and
		to the remainders $E_\lambda^\pm$, therefore yields
		\[
		\norm{\mathcal D_\lambda(t,s)}_{L^2(\R^{n-3})\to L^2(\R^{n-3})}
		\leq C\frac{\lambda}{|t-s|}(\lambda|t-s|)^{-1}
		=\frac{C}{|t-s|^2}.
		\]
		The operator-valued Schur test and the fixed finite directional sum give
		\[
		\norm{\mathcal E_\lambda}_{L^2(\R^{n-2})\to L^2(\R^{n-2})}
		\leq C\lambda+C\int_{C_0/\lambda\leq|r|\leq r_0}\frac{\dd r}{r^2}
		\leq C\lambda.\qedhere
		\]
	\end{proof}

\subsection{Singular Radon transforms}
\label{sec:upper-polynomial}

The next lemma is a consequence of Stein--Street's estimate \cite{SS} on singular Radon transforms.  For related polynomial and finite-type results, see
\cite{CNSW,CRW,SteinHA}.

\begin{lemma}\label{upper:polynomial}
Let $d\ge0,\ k\ge1$, let $\Theta\subset\R^d$ be compact, and fix
$\psi_0,\psi_1\in C_c^\infty(\R^{k+1})$.
Suppose $\gamma_\theta(t,\tau,\xi)\in\R^k$ is jointly real analytic
near
\[
\{(\theta,t,0,\xi):\theta\in\Theta,\ (t,\xi)\in\supp\psi_0\},
\]
with $\gamma_\theta(t,0,\xi)=\xi$ throughout that neighborhood.
There are $\rho\in(0,1)$ and $C<\infty$ such that
\begin{equation}\label{upper:polynomial-operator}
\begin{aligned}
[\mathcal P_{\theta,\epsilon}F](t,\xi)
={}&\psi_0(t,\xi)\int_\R
\frac{\chi_1(2\tau/\rho)(1-\chi_1(\tau/\epsilon))}{\tau}\,
(\psi_1F)(t-\tau,\gamma_\theta(t,\tau,\xi))\dd\tau
\end{aligned}
\end{equation}
satisfies
\begin{equation}\label{upper:polynomial-bound}
\norm{\mathcal P_{\theta,\epsilon}F}_{L^2(\R^{k+1})}
\leq C\norm F_{L^2(\R^{k+1})},
\qquad \theta\in\Theta,\quad\epsilon>0.
\end{equation}
The constants depend only on the analytic family, $\Theta$, and the
cutoffs. If $d=0$, omit $\theta$. 
\end{lemma}

\begin{proof}
Choose a bounded open neighborhood $\Theta'$ of $\Theta$ on whose
closure the same analyticity holds for sufficiently small $|\tau|$.
The map
\[
(\theta,t,\xi)\longmapsto
(\theta,t-\tau,\gamma_\theta(t,\tau,\xi))
\]
is jointly analytic and equals the identity at $\tau=0$.
For smooth $U$ supported in $\Theta'\times\R^{k+1}$, define
\[
[\mathcal P_0U](\theta,t,\xi)
=\psi_0(t,\xi)\pv\int_\R\frac{\chi_1(2\tau/\rho)}\tau
(\psi_1U)(\theta,t-\tau,\gamma_\theta(t,\tau,\xi))\dd\tau.
\]
A finite localization on
$\overline{\Theta'}\times\supp\psi_0$, with input cutoffs equal to one
on the corresponding images, and for sufficiently small $\rho$,  \cite[Corollary~5.6]{SS} give
\[
\norm{\mathcal P_0U}_{L^2(\Theta'\times\R^{k+1})}
\leq C\norm U_{L^2(\Theta'\times\R^{k+1})}.
\]

Let $\mathcal P_\epsilon$ be obtained by inserting
$1-\chi_1(\tau/\epsilon)$ in the integral. Fourier inversion and
$e^{ivt/\epsilon}e^{-iv(t-\tau)/\epsilon}=e^{iv\tau/\epsilon}$ give
\begin{equation}\label{upper:truncation}
\mathcal P_\epsilon U
=\mathcal P_0U-\frac1{2\pi}\int_\R\widehat{\chi_1}(v)
 e^{ivt/\epsilon}\mathcal P_0(e^{-ivt/\epsilon}U)\dd v.
\end{equation}
 Hence
\[
\norm{\mathcal P_\epsilon U}_{L^2(\Theta'\times\R^{k+1})}
\leq C\left(1+\frac{\norm{\widehat{\chi_1}}_{L^1(\R)}}{2\pi}\right)
\norm U_{L^2(\Theta'\times\R^{k+1})}.
\]

For $U(\theta,t,\xi)=h(\theta)F(t,\xi)$, this yields
\[
\int_{\Theta'}|h(\theta)|^2
\norm{\mathcal P_{\theta,\epsilon}F}_{L^2(\R^{k+1})}^2\dd\theta
\leq C\norm h_{L^2(\Theta')}^2\norm F_{L^2(\R^{k+1})}^2.
\]
For fixed $\epsilon>0$ and smooth $F$, the truncated integral is
continuous in $\theta$ as an $L^2$-valued function. Given
$\theta_0\in\Theta$, choose $h_j\in C_c^\infty(\Theta')$ with
\[
\norm{h_j}_{L^2(\Theta')}=1,
\qquad\supp h_j\subset\{|\theta-\theta_0|<j^{-1}\}.
\]
Letting $j\to\infty$ proves \eqref{upper:polynomial-bound} at
$\theta_0$. Density completes the proof. For $d=0$, the parameter
localization and the last step are omitted.
\end{proof}

\begin{proposition}\label{upper:transport}
Let $k\ge1$, and let $a(t,\tau,\xi)\in\mathbb C$ and
$\gamma(t,\tau,\xi)\in\R^k$ be fixed smooth functions. Assume that,
for $|\tau|\leq r_0$, $a$ vanishes outside a fixed compact set
$K\subset\R^{k+1}$ in $(t,\xi)$, and
\[
\gamma(t,0,\xi)=\xi,
\qquad
\norm{\partial_\xi\gamma(t,\tau,\xi)-I}_{\R^k\to\R^k}
\leq\tfrac12.
\]
Define
\begin{equation}\label{upper:Radon}
[\mathcal R_\lambda F](t,\xi)
=\int_\R k_\lambda(\tau)a(t,\tau,\xi)
F(t-\tau,\gamma(t,\tau,\xi))\dd\tau.
\end{equation}
There is $B$ independent of $m$ such that, for every $m\geq2$,
\begin{equation}\label{upper:transportestimate}
\begin{aligned}
\norm{\mathcal R_\lambda F}_{L^2(\R^{k+1})}
\leq{}&\left(C_m+\frac Bm\log\lambda\right)
\norm F_{L^2(\R^{k+1})}
+C_m\lambda^{-1}\norm{\nabla_\xi F}_{L^2(\R^{k+1})},
\qquad\lambda\geq2.
\end{aligned}
\end{equation}
If $\gamma$ is real analytic near
$\{(t,0,\xi):(t,\xi)\in K\}$, then
\begin{equation}\label{upper:Radon-analytic}
\norm{\mathcal R_\lambda F}_{L^2(\R^{k+1})}
\leq C\norm F_{L^2(\R^{k+1})},\qquad\lambda\geq2.
\end{equation}
\end{proposition}

\begin{proof}
By density, we take $F\in C_c^\infty(\R^{k+1})$ and fix $m\geq2$.
Write
\[
G(t,\tau,\xi)=\int_0^1\partial_\tau\gamma(t,u\tau,\xi)\dd u,
\qquad \gamma(t,\tau,\xi)=\xi+\tau G(t,\tau,\xi).
\]
Choose $c_m>0$ sufficiently small, independently of $\lambda$, and let
\[
\delta=c_m\lambda^{-1/m},\qquad 0<c_m<r_0/4.
\]
Require also $\chi_0=1$ on $[-2c_m,2c_m]$.
Let $\mathcal R_{\lambda,\delta}$ be \eqref{upper:Radon} with the
additional factor $\chi_1(\tau/\delta)$.

\emph{Polynomial models.}
Let $Q=[-1/2,1/2)^{k+1}$ and let
$Q_j=(t_j,\xi_j)+\delta Q$ be the disjoint grid cubes meeting $K$.
Define the Taylor polynomials
\[
G_j(t,\tau,\xi)
=\sum_{|\mu|\leq m-2}
\frac{(t-t_j,\tau,\xi-\xi_j)^\mu}{\mu!}
\partial^\mu G(t_j,0,\xi_j),
\qquad \mu\in\mathbb N_0^{k+2}.
\]
For $(t,\xi)\in Q_j$, $|\tau|\leq2\delta$, and $0\leq u\leq1$, let
\[
\gamma_{j,u}(t,\tau,\xi)
=\xi+\tau\bigl((1-u)G_j(t,\tau,\xi)+uG(t,\tau,\xi)\bigr).
\]
Taylor's formula and a sufficiently small $c_m$ give
\begin{equation}\label{upper:Taylor}
\begin{gathered}
|G-G_j|\leq C_m\delta^{m-1},\quad
\norm{\partial_\xi\gamma_{j,u}-I}_{\R^k\to\R^k}
\leq C_m|\tau|\leq\tfrac12,\quad
|\gamma_{j,u}-\xi|\leq C_m\delta.
\end{gathered}
\end{equation}
The maps $(t,\xi)\mapsto(t-\tau,\gamma_{j,u})$ are injective on $Q_j$ and the Jacobian  satisfies
\[
\left|\det \partial_{(t,\xi)}(t-\tau,\gamma_{j,u})\right|
=|\det\partial_\xi\gamma_{j,u}|\geq2^{-k}.
\]
Thus their images lie in fixed concentric enlargements $Q_j^*$, with
$\sum_j\ind_{Q_j^*}\leq C_m$, and change of variables gives
\begin{equation}\label{upper:composition}
\norm{F(t-\tau,\gamma_{j,u}(t,\tau,\xi))}_{L^2(Q_j)}
\leq2^{k/2}\norm F_{L^2(Q_j^*)}.
\end{equation}
The same estimate holds with $F$ replaced by $\nabla_\xi F$.

On $Q_j$, define
\[
[\mathcal P_jF](t,\xi)
=\int_\R k_\lambda(\tau)\chi_1(\tau/\delta)
F(t-\tau,\xi+\tau G_j(t,\tau,\xi))\dd\tau.
\]
Define the coefficient vectors
\[
\theta_\mu
=\frac{\delta^{|\mu|}}{\mu!}\partial^\mu G(t_j,0,\xi_j)
\in\R^k,
\qquad
\theta=(\theta_\mu)_{|\mu|\leq m-2}.
\]
The rescaled polynomial is
\[
G_\theta(s,\tau,\xi)
:=\sum_{|\mu|\leq m-2}\theta_\mu(s,\tau,\xi)^\mu
=G_j(t_j+\delta s,\delta\tau,\xi_j+\delta\xi).
\]
For fixed $m$, these coefficient vectors lie in a compact box
independent of $j$ and $\lambda$. Viewing $\theta$ as an independent
parameter, the map
\[
(\theta,s,\tau,\xi)\longmapsto\xi+\tau G_\theta(s,\tau,\xi)
\]
is polynomial.
Choose fixed $\psi_0,\psi_1$ equal to one on $Q$ and, respectively,
on all its images for these coefficients and $|\tau|\leq2$.
Enlarge $Q_j^*$ so its rescaling contains $\supp\psi_1$, and set
\[
F_j(s,\xi)=\delta^{(k+1)/2}(\ind_{Q_j^*}F)
(t_j+\delta s,\xi_j+\delta\xi),\qquad
\norm{F_j}_{L^2(\R^{k+1})}=\norm F_{L^2(Q_j^*)}.
\]
For $(s,\xi)\in Q$, the rescaled operator is
\[
\delta^{(k+1)/2}[\mathcal P_jF](t_j+\delta s,\xi_j+\delta\xi)=\int_\R
\frac{\chi_1(\tau)(1-\chi_1(\lambda\delta\tau/C_0))}{\tau}
(\psi_1F_j)(s-\tau,\xi+\tau G_\theta(s,\tau,\xi))\dd\tau.
\]
Let $\rho_m$ be supplied by Lemma~\ref{upper:polynomial}.
Then the part with $\chi_1(\tau)$ replaced by $\chi_1(2\tau/\rho_m)$
has norm  $\le C_m\norm{F_j}_{L^2(\R^{k+1})}$, uniformly for
$\epsilon=C_0/(\lambda\delta)>0$.
The remaining part has $\rho_m/2\leq|\tau|\leq2$.
So \eqref{upper:composition}
and Minkowski's inequality give
\begin{equation}\label{upper:local-polynomial-bound}
\begin{aligned}
\norm{\mathcal P_jF}_{L^2(Q_j)}
&\leq\left(C_m+2^{k/2}
\int_{\rho_m/2\leq|\tau|\leq2}\frac{\dd\tau}{|\tau|}\right)
\norm F_{L^2(Q_j^*)}
\leq C_m\norm F_{L^2(Q_j^*)}.
\end{aligned}
\end{equation}

\emph{Taylor error.}
Since $|a(t,\tau,\xi)-a(t,0,\xi)|\leq C|\tau|$ and
\[
\begin{aligned}
&F(t-\tau,\gamma(t,\tau,\xi))
-F(t-\tau,\xi+\tau G_j(t,\tau,\xi))\\
&\quad=\tau\int_0^1(G-G_j)(t,\tau,\xi)\cdot
\nabla_\xi F(t-\tau,\gamma_{j,u}(t,\tau,\xi))\dd u,
\end{aligned}
\]
\eqref{upper:Taylor}--\eqref{upper:composition} and
$|k_\lambda(\tau)|\leq|\tau|^{-1}$ imply
\begin{equation}\label{upper:Taylorerror}
\begin{aligned}
\norm{\mathcal R_{\lambda,\delta}F
-a(t,0,\xi)\mathcal P_jF}_{L^2(Q_j)}&\leq C_m\int_{|\tau|\leq2\delta}
\left(\norm F_{L^2(Q_j^*)}
+\delta^{m-1}\norm{\nabla_\xi F}_{L^2(Q_j^*)}\right)\dd\tau\\
&\leq C_m\left(\delta\norm F_{L^2(Q_j^*)}
+\delta^m\norm{\nabla_\xi F}_{L^2(Q_j^*)}\right).
\end{aligned}
\end{equation}
Squaring, summing over the disjoint $Q_j$, and using
\eqref{upper:local-polynomial-bound} and bounded overlap yield
\begin{equation}\label{upper:near}
\norm{\mathcal R_{\lambda,\delta}F}_{L^2(\R^{k+1})}
\leq C_m\left(\norm F_{L^2(\R^{k+1})}
+\delta^m\norm{\nabla_\xi F}_{L^2(\R^{k+1})}\right).
\end{equation}

\emph{Far distances.}
For the original fixed map, change of variables gives
\[
\norm{a(t,\tau,\xi)F(t-\tau,\gamma(t,\tau,\xi))}_{L^2(\R^{k+1})}
\leq C\norm F_{L^2(\R^{k+1})},
\]
where $C$ is independent of $m$. Consequently,
\begin{equation}\label{upper:far}
\begin{aligned}
\norm{\mathcal R_\lambda-\mathcal R_{\lambda,\delta}}
_{L^2(\R^{k+1})\to L^2(\R^{k+1})}
\leq2C\int_\delta^{r_0}\frac{\dd\tau}{\tau}=\frac{2C}{m}\log\lambda+2C\log\frac{r_0}{c_m}.
\end{aligned}
\end{equation}
Since $\delta^m=c_m^m\lambda^{-1}\leq\lambda^{-1}$,
\eqref{upper:near} proves \eqref{upper:transportestimate} with $B=2C$.

If $\gamma$ is real analytic, replace $a(t,\tau,\xi)$ by
$a(t,0,\xi)$ as above, at a uniformly bounded cost.
Instead of taking Taylor polynomials, apply
Lemma~\ref{upper:polynomial} directly to $\gamma$, with $d=0$,
$\epsilon=C_0/\lambda$. This uniformly bounds the part
with cutoff $\chi_1(2\tau/\rho)$, where $\rho>0$ is fixed
and sufficiently small that $\chi_0=1$ on $[-\rho,\rho]$.
The remaining part is supported where
$\rho/2\leq|\tau|\leq r_0$ and is bounded, as in
\eqref{upper:far}, by
$2C\log\frac{2r_0}{\rho}.$
This proves
\eqref{upper:Radon-analytic}.  
\end{proof}

\subsection{The Bargmann transform}
\label{sec:upper-gaussian}

Extend $\phi$ compactly, unchanged near $\supp b$, and shrink $r_0$
so that \eqref{upper:phase} holds globally in $(z,\eta)$.
The extension and $r_0$ are fixed independently of $m,\lambda$.

\begin{lemma}[Uniform transverse bound]\label{upper:uniform-V}
For $V_\lambda$ in \eqref{upper:Vdef}--\eqref{upper:phase} and an integer $N>n-3$,
\begin{equation}\label{upper:Vnorm}
\norm{V_\lambda(t,s)}_{L^2(\R^{n-3})\to L^2(\R^{n-3})}
\leq C\norm{b(t,s,\cdot,\cdot)}_{C^{2N}(\R^{2n-6})},
\end{equation}
uniformly for $\lambda\geq1$ and $|t-s|\leq r_0$.
\end{lemma}

\begin{proof}
For $n=3$, $|V_\lambda(t,s)|=|b(t,s)|$.
For $n\geq4$, suppress $t,s$. The kernel of
$V_\lambda(t,s)V_\lambda(t,s)^*$ is
\[
\left(\frac\lambda{2\pi}\right)^{n-3}
\int e^{i\lambda(\Phi(z,\eta)-\Phi(z',\eta))}
b(z,\eta)\overline{b(z',\eta)}\dd\eta.
\]
The  bound \eqref{upper:phase} implies
\[
|\Phi_\eta(z,\eta)-\Phi_\eta(z',\eta)|\geq\tfrac12|z-z'|,
\qquad
|\partial_\eta^\alpha(\Phi(z,\eta)-\Phi(z',\eta))|
\leq C_\alpha|z-z'|\quad(|\alpha|\geq1).
\]
Integration by parts bounds the kernel by
\[
C_N\norm b_{C^{2N}(\R^{2n-6})}^2
\lambda^{n-3}(1+\lambda|z-z'|)^{-N}.
\]
For $N>n-3$, its  integrals in $z$ and $z'$ are 
$\le C\norm b_{C^{2N}(\R^{2n-6})}^2$. Schur's test proves the claim.
\end{proof}

We use a rescaled Bargmann transform. Its isometry is classical
\cite[Proposition~2.1]{TataruPhaseSpace} and  the localized derivative bound
below follows from the same calculation.
	For $n\geq4$, $\lambda\geq1$, $h\in L^2(\R^{n-3})$ and
	$\xi=(q,p)\in\R^{2n-6}$, where $q,p\in\R^{n-3}$, define
	\begin{equation}\label{upper:Wdef}
		\W h(\xi)=c_\lambda\int_{\R^{n-3}}
		e^{-\lambda|w-q|^2/2-i\lambda p\cdot(w-q)}h(w)\dd w,
		\qquad
		c_\lambda=\left(\frac\lambda{2\pi}\right)^{\frac{n-3}{2}}
		\left(\frac\lambda\pi\right)^{\frac{n-3}{4}}.
	\end{equation}
	
	\begin{lemma}[Bargmann transform]\label{upper:gaussian-estimate}
		The transform in \eqref{upper:Wdef} satisfies
		\[
		\norm{\W h}_{L^2(\R^{2n-6})}=\norm h_{L^2(\R^{n-3})},
		\qquad \W^*\W=I.
		\]
		For every fixed $\varphi\in C_c^\infty(\R^{2n-6})$ and every
		$h\in L^2(\R^{n-3})$,
		\begin{equation}\label{upper:Wbound}
			\norm{\varphi\W h}_{L^2(\R^{2n-6})}
			+\lambda^{-1}\norm{\nabla_\xi(\varphi\W h)}_{L^2(\R^{2n-6})}
			\leq C_{\varphi}\norm{h}_{L^2(\R^{n-3})}.
		\end{equation}
		The constant $C_{\varphi}$ is independent of $\lambda\geq1$.
	\end{lemma}

\begin{proof}
For smooth $h$, Plancherel in $p$ and then integration in $q$ give
\[
\begin{aligned}
\int_{\R^{n-3}}|\W h(q,p)|^2\dd p
&=\left(\frac\lambda\pi\right)^{(n-3)/2}
\int_{\R^{n-3}}e^{-\lambda|w-q|^2}|h(w)|^2\dd w,\\
\norm{\W h}_{L^2(\R^{2n-6})}^2
&=\norm h_{L^2(\R^{n-3})}^2.
\end{aligned}
\]
Thus $\W^*\W=I$ and $\norm{\W^*}_{L^2(\R^{2n-6})\to L^2(\R^{n-3})}=1$.
Differentiation and the same  calculation yield
\[
\begin{gathered}
\partial_{p_j}\W h(q,p)
=-i\lambda c_\lambda\int (w_j-q_j)
 e^{-\lambda|w-q|^2/2-i\lambda p\cdot(w-q)}h(w)\dd w,\\
\partial_{q_j}\W h=i\lambda p_j\W h+i\partial_{p_j}\W h,\\
\norm{\partial_{p_j}\W h}_{L^2(\R^{2n-6})}^2
=\lambda^2\left(\frac\lambda\pi\right)^{(n-3)/2}
\int |h(w)|^2\int (w_j-q_j)^2e^{-\lambda|w-q|^2}\dd q\dd w
=\frac\lambda2\norm h_{L^2(\R^{n-3})}^2.
\end{gathered}
\]
On $\supp\varphi$, $p$ is bounded. The product rule therefore gives
\eqref{upper:Wbound}. Density extends the identities and estimates
to $h\in L^2(\R^{n-3})$.
\end{proof}

	Let $V_\lambda(t,s)$ be given by \eqref{upper:Vdef}--\eqref{upper:phase},
	with $n\geq4$ and the fixed global extension of $\phi$ above.
	For each $\xi$, let $p_*$ be the unique solution of
	\[
	p=\Phi_z(t,s,q,p_*)=p_*+(t-s)\phi_z(t,s,q,p_*).
	\]
	The bound $\norm{\Phi_{z\eta}-I}_{\R^{n-3}\to\R^{n-3}}<1/2$
	makes this a global smooth change of variables. Define
	\begin{equation}\label{upper:map}
		\begin{aligned}
			q_*&=\Phi_\eta(t,s,q,p_*)=q+(t-s)\phi_\eta(t,s,q,p_*),\\
			\gamma(t,s,\xi)&=(q_*,p_*),\quad S(t,s,\xi)=\Phi(t,s,q,p_*)-q_*\cdot p_*,
		\end{aligned}
	\end{equation}
	Writing $b(t,s,\xi)=b(t,s,q,p)$, for $F\in L^2(\R^{2n-6})$ set
	\[
	[U_\lambda(t,s)F](\xi)
	=b(t,s,\xi)e^{i\lambda S(t,s,\xi)}F(\gamma(t,s,\xi)).
	\]
	
The construction follows Tataru's phase-space representation
using the Bargmann transform
\cite[Theorem~4 and Section~7.1]{TataruPhaseSpace}.
We prove the near-identity version needed here, with the original
amplitude $b$ and the explicit error in \eqref{upper:factor}.

\begin{lemma}[Phase-space representation]\label{upper:gaussian-factorization}
		After decreasing $r_0$, for $\lambda\geq1$ and $|t-s|\leq r_0$,
		\[
		\begin{gathered}
		\gamma(t,t,\xi)=\xi,\qquad S(t,t,\xi)=0,\\
		\norm{\partial_\xi\gamma(t,s,\xi)-I}_{\R^{2n-6}\to\R^{2n-6}}
		\leq C|t-s|\leq\tfrac12,
		\end{gathered}
		\]
		and
		\begin{equation}\label{upper:factor}
			\begin{gathered}
				V_\lambda(t,s)=\W^*U_\lambda(t,s)\W+E_\lambda(t,s),\\
				\norm{E_\lambda(t,s)}_{L^2(\R^{n-3})\to L^2(\R^{n-3})}
				\leq C\bigl(|t-s|+\lambda^{-1/2}\bigr).
			\end{gathered}
		\end{equation}
		The constant $C$ is independent of $t,s,\lambda$.
	\end{lemma}
	
	\begin{proof}
		Fix $t,s$ and suppress them in $\Phi,\phi,b$.
		Implicit differentiation gives uniformly bounded first derivatives of $p_*$.
		Since
		\[
		\begin{aligned}
		\gamma(t,s,\xi)-\xi&=(t-s)(\phi_\eta,-\phi_z)(q,p_*),\\
		S&=(t-s)(\phi-p_*\cdot\phi_\eta)(q,p_*),
		\end{aligned}
		\]
		the assertions about $\gamma$ and $S$ follow.
		
		Since 
		\[[U_\lambda F](q,p)
		=b(q,p)e^{i\lambda S}F(q_*,p_*),\]
		we have
		\[[\W^*U_\lambda\W h](z)
			=c_\lambda\iint
			e^{-\lambda|z-q|^2/2+i\lambda p\cdot(z-q)}
			b(q,p)e^{i\lambda S}
			[\W h](q_*,p_*)\,dq\,dp,\]
		where 
		\[[\W h](q_*,p_*)=c_\lambda\int
			e^{-\lambda|w-q_*|^2/2-i\lambda p_*\cdot(w-q_*)}
			h(w)\,dw.\]
		Inserting the Gaussian Fourier identity
		\[
		e^{-\lambda|w-q_*|^2/2-i\lambda p_*\cdot(w-q_*)}
		=\left(\frac\lambda{2\pi}\right)^{\frac{n-3}{2}}
		\int e^{-\lambda|\eta-p_*|^2/2-i\lambda(w-q_*)\cdot\eta}\dd\eta
		\]
		gives
	\[[\W^*U_\lambda\W h](z)
	=c_\lambda^2
			\left(\frac{\lambda}{2\pi}\right)^{\frac{n-3}{2}}
			\iiiint
			e^{-\lambda(|z-q|^2+|\eta-p_*|^2)/2}
			e^{i\lambda[p\cdot(z-q)+S-(w-q_*)\cdot\eta]}
			b(q,p)h(w)\,dw\,d\eta\,dq\,dp.\]
		
		For each fixed \(q\), change variables from \(p\) to \(p_*\) through \(p=\Phi_z(q,p_*).\)
		We have \[dp=\det\Phi_{z\eta}(q,p_*)\,dp_*,\]
		\[
		\begin{aligned}
			p\cdot(z-q)+S-(w-q_*)\cdot\eta
			=\Phi(q,p_*)+\Phi_z(q,p_*)\cdot(z-q)+\Phi_\eta(q,p_*)\cdot(\eta-p_*)-w\cdot\eta.
		\end{aligned}
		\]
		For fixed $(z,\eta)$, introduce new variables $u_1,u_2\in\R^{n-3}$ by
		\[
		q=z+\lambda^{-1/2}u_1,\qquad p_*=\eta+\lambda^{-1/2}u_2,
		\qquad \dd q\dd p_*=\lambda^{-(n-3)}\dd u_1\dd u_2.
		\]
		Taylor's formula gives
		\[
		\begin{aligned}
			&\lambda\bigl[\Phi(q,p_*)+\Phi_z(q,p_*)\cdot(z-q)
			+\Phi_\eta(q,p_*)\cdot(\eta-p_*)-\Phi(z,\eta)\bigr]\\
			&\qquad=-u_1\cdot u_2+\omega_\lambda(z,\eta,u_1,u_2),\\
			&\omega_\lambda=-(t-s)\int_0^1\tau
			(u_1\cdot\partial_z+u_2\cdot\partial_\eta)^2
			\phi(z+\tau\lambda^{-1/2}u_1,\eta+\tau\lambda^{-1/2}u_2)\dd\tau.
		\end{aligned}
		\]
		Since
		\[
		c_\lambda^2\left(\frac\lambda{2\pi}\right)^{\frac{n-3}{2}}
		\lambda^{-(n-3)}
		=\left(\frac\lambda{2\pi}\right)^{n-3}
		\frac1{2^{\frac{n-3}{2}}\pi^{n-3}},
		\]
		$\W^*U_\lambda\W$ has the phase in \eqref{upper:Vdef} and amplitude
		\begin{equation}\label{upper:mixedamplitude}
			\frac1{2^{\frac{n-3}{2}}\pi^{n-3}}
			\iint_{\R^{n-3}\times\R^{n-3}}
			e^{-(|u_1|^2+|u_2|^2)/2-iu_1\cdot u_2+i\omega_\lambda}
			b(q,\Phi_z(q,p_*))\det\Phi_{z\eta}(q,p_*)\dd u_1\dd u_2.
		\end{equation}
		Moreover,
		\[
		\iint e^{-(|u_1|^2+|u_2|^2)/2-iu_1\cdot u_2}\dd u_1\dd u_2
		=(2\pi)^{\frac{n-3}{2}}\int e^{-|u_2|^2}\dd u_2
		=2^{\frac{n-3}{2}}\pi^{n-3}.
		\]
		Since the  amplitude inside \eqref{upper:mixedamplitude} is close to $b(z,\eta)$, we  define the error amplitude
		\[
		a_{\lambda,u_1,u_2}(z,\eta)
		=e^{i\omega_\lambda}b(q,\Phi_z(q,p_*))
		\det\Phi_{z\eta}(q,p_*)-b(z,\eta).
		\]
		All following $C^{2n}$ norms are in $(z,\eta)$, with $u_1,u_2$ fixed.
		The identities $\Phi_z=p_*+(t-s)\phi_z$ and
		$\Phi_{z\eta}=I+(t-s)\phi_{z\eta}$ give
		\[
		\begin{aligned}
			\norm{\omega_\lambda}_{C^{2n}}&\leq C|t-s|(|u_1|+|u_2|)^2,\\
			\norm{\det\Phi_{z\eta}(q,p_*)-1}_{C^{2n}}&\leq C|t-s|,\\
			\norm{b(q,\Phi_z(q,p_*))-b(z,\eta)}_{C^{2n}}
			&\leq C\bigl(|t-s|+\lambda^{-1/2}(|u_1|+|u_2|)\bigr).
		\end{aligned}
		\]
		Using $e^{i\omega_\lambda}-1
		=i\omega_\lambda\int_0^1e^{i\tau\omega_\lambda}\dd\tau$
		and the product rule yields
		\begin{equation}\label{upper:amplitudeerror}
			\norm{a_{\lambda,u_1,u_2}}_{C^{2n}}
			\leq C\bigl(|t-s|+\lambda^{-1/2}\bigr)
			(1+|u_1|+|u_2|)^{4n+2}.
		\end{equation}
		The $\eta$-support of $a_{\lambda,u_1,u_2}$ has uniformly bounded measure,
		so the $TT^*$ proof of Lemma~\ref{upper:uniform-V} applies uniformly. Applying Lemma \ref{upper:uniform-V}  before integration in $u_1,u_2$ gives
		\[
		\begin{aligned}
			&\norm{\W^*U_\lambda(t,s)\W-V_\lambda(t,s)}
			_{L^2(\R^{n-3})\to L^2(\R^{n-3})}\\
			&\quad\leq C\iint e^{-(|u_1|^2+|u_2|^2)/2}
			\norm{a_{\lambda,u_1,u_2}}_{C^{2n}}\dd u_1\dd u_2
			\leq C\bigl(|t-s|+\lambda^{-1/2}\bigr).
		\end{aligned}
		\]
		Taking $E_\lambda(t,s)=V_\lambda(t,s)-\W^*U_\lambda(t,s)\W$
		proves \eqref{upper:factor}.
	\end{proof}

\subsection{Completion of the proof}
\label{sec:upper-completion}

\begin{proposition}\label{upper:oscillatory}
Let $H_\lambda$ be given by \eqref{upper:Hdef}, with fixed smooth
data in \eqref{upper:Vdef}--\eqref{upper:phase}.
There is $B_0$, independent of $m$, such that, for every integer $m\geq2$,
\begin{equation}\label{upper:full}
\norm{H_\lambda}_{L^2(\R^{n-2})\to L^2(\R^{n-2})}
\leq C_m+\frac{B_0}{m}\log\lambda,
\qquad m\geq2,\quad\lambda\geq2.
\end{equation}
If $\phi$ is real analytic on a fixed neighborhood of the relevant
amplitude support, then
\begin{equation}\label{upper:H-analytic}
\norm{H_\lambda}_{L^2(\R^{n-2})\to L^2(\R^{n-2})}\leq C.
\end{equation}
\end{proposition}

\begin{proof}
By \eqref{upper:Vnorm}, the part with
$\rho\leq|t-s|\leq r_0$, for fixed $\rho>0$, has norm at most
$2C\log(r_0/\rho)$. We may therefore decrease $r_0$ and shrink the fixed support of
$\chi_0$ whenever needed below, retaining the notation $k_\lambda$.

Suppose $n\geq4$. Use $\gamma,S,U_\lambda$ from
Lemma~\ref{upper:gaussian-factorization}, and define
\[
[\mathcal U_\lambda F](t,\xi)
=\int_\R k_\lambda(t-s)[U_\lambda(t,s)F(s,\cdot)](\xi)\dd s.
\]
Choose a fixed $\varphi\in C_c^\infty(\R^{2n-6})$ equal to one on
all input images $\gamma(t,s,\xi)$ with $(t,s,\xi)\in\supp b$.
For $f\in C_c^\infty(\R^{n-2})$, let
\[
F(t,\xi)=\varphi(\xi)\W(f(t,\cdot))(\xi),\qquad
F_1(t,\xi,u)=\chi_1(u/4)e^{i\lambda u}F(t,\xi),\quad u\in\R.
\]
Lemma~\ref{upper:gaussian-estimate} and differentiation in $u$ give
\begin{equation}\label{upper:lifted-input}
\norm{F_1}_{L^2(\R^{2n-4})}
+\lambda^{-1}\norm{\nabla_{\xi,u}F_1}_{L^2(\R^{2n-4})}
\leq C\norm f_{L^2(\R^{n-2})}.
\end{equation}
Apply Proposition~\ref{upper:transport} with $k=2n-5$, spatial
variable $(\xi,u)$, amplitude and map
\begin{equation}\label{upper:lifted-map}
\begin{aligned}
a_1(t,\tau,\xi,u)&=\chi_1(u)b(t,t-\tau,\xi),\\
\gamma_1(t,\tau,\xi,u)
&=\bigl(\gamma(t,t-\tau,\xi),u+S(t,t-\tau,\xi)\bigr).
\end{aligned}
\end{equation}
These data are independent of $\lambda,m$. By Lemma \ref{upper:gaussian-factorization},
\[
\gamma_1(t,0,\xi,u)=(\xi,u),\qquad
\partial_{(\xi,u)}\gamma_1=
\begin{pmatrix}
\partial_\xi\gamma&0\\ \partial_\xi S&1
\end{pmatrix}
=I+O(|\tau|).
\]
After a fixed shrinking, all the hypotheses of
Proposition~\ref{upper:transport} hold and $|S|\leq1$ on the
amplitude support. There $|u|\leq2$ implies $|u+S|\leq3$, so
$\chi_1((u+S)/4)=1$. The operator \eqref{upper:Radon} for
$a_1,\gamma_1$ consequently satisfies the exact identity
\begin{equation}\label{upper:lifted-identity}
[\mathcal R_\lambda F_1](t,\xi,u)
=\chi_1(u)e^{i\lambda u}[\mathcal U_\lambda F](t,\xi).
\end{equation}
Then \eqref{upper:transportestimate} and
\eqref{upper:lifted-input}  give
\begin{equation}\label{upper:transport-application}
\begin{aligned}
&\norm{\chi_1}_{L^2(\R)}
\norm{\mathcal U_\lambda F}_{L^2(\R^{2n-5})}
=\norm{\mathcal R_\lambda F_1}_{L^2(\R^{2n-4})}\\
&\le 
	\Big(C_m+\frac Bm\log\lambda\Big)
	\|F_1\|_{L^2(\mathbb R^{2n-4})}+
	C_m\lambda^{-1}
	\|\nabla_{\xi,u}F_1\|_{L^2(\mathbb R^{2n-4})}\\
&\leq\left(C_m+\frac Cm\log\lambda\right)
\norm f_{L^2(\R^{n-2})}.
\end{aligned}
\end{equation}
The choice of $\varphi$ and \eqref{upper:factor} imply
\[
H_\lambda f(t,\cdot)
=\W^*((\mathcal U_\lambda F)(t,\cdot))
+\int_\R k_\lambda(t-s)E_\lambda(t,s)f(s,\cdot)\dd s.
\]
Lemma \ref{upper:gaussian-factorization} and the operator-valued Schur test give
\begin{equation}\label{upper:Gaussian-error}
\begin{aligned}
&\norm{H_\lambda f-\W^*(\mathcal U_\lambda F)}_{L^2(\R^{n-2})}\\
&\quad\leq C\int_{C_0/\lambda\leq|\tau|\leq r_0}
\left(1+\frac{\lambda^{-1/2}}{|\tau|}\right)\dd\tau\,
\norm f_{L^2(\R^{n-2})}\\
&\quad\leq C\bigl(1+\lambda^{-1/2}\log(2+\lambda)\bigr)
\norm f_{L^2(\R^{n-2})}\leq C\norm f_{L^2(\R^{n-2})}.
\end{aligned}
\end{equation}
Since $\norm{\W^*}_{L^2(\R^{2n-6})\to L^2(\R^{n-3})}=1$,
this proves \eqref{upper:full}.

If $\phi$ is analytic near the amplitude support, the analytic
implicit function theorem applied to
$p=p_*+(t-s)\phi_z(t,s,q,p_*)$ makes $p_*$, and hence $\gamma,S$,
analytic there. The fixed neighborhood contains $(q,p_*)$ for
sufficiently small $|t-s|$, since $|p_*-p|\leq C|t-s|$.
Thus $\gamma_1$ is analytic near its output support.
We use \eqref{upper:Radon-analytic} instead of
\eqref{upper:transportestimate} in \eqref{upper:transport-application},
then \eqref{upper:Gaussian-error} proves \eqref{upper:H-analytic}.
The amplitudes and cutoffs need only be smooth.

For $n=3$, we use the same argument without the Bargmann transform:
\[
\begin{gathered}
F_1(t,u)=\chi_1(u/4)e^{i\lambda u}f(t),\\
a_1(t,\tau,u)=\chi_1(u)b(t,t-\tau),\qquad
\gamma_1(t,\tau,u)=u+\Phi(t,t-\tau).
\end{gathered}
\]
Here $\mathcal R_\lambda F_1=\chi_1(u)e^{i\lambda u}H_\lambda f(t)$,
and Proposition~\ref{upper:transport} applies with $k=1$.
Density completes both estimates.
\end{proof}

\begin{proof}[Proof of Theorem~\ref{thm:general}]
The $TT^*$ identity, Proposition~\ref{upper:transverse-factorization},
and Proposition~\ref{upper:oscillatory} give
\begin{equation}\label{upper:localTT}
\norm{\beta T_\lambda}_{L^2(M)\to L^2(\Sigma)}^2
\leq\lambda\left(C_{\beta,m}+\frac{B_\beta}{m}\log\lambda\right).
\end{equation}
The coefficient $B_\beta$ is independent of $m$, since the
directional decomposition is fixed. Summing over the fixed
partition $\sum_\beta\beta^2=1$ yields
\[
\begin{aligned}
\norm{e_\lambda}_{L^2(\Sigma)}^2
=\sum_\beta\norm{\beta T_\lambda e_\lambda}_{L^2(\Sigma)}^2\leq\left(A_m\lambda+\frac Bm\lambda\log\lambda\right)
\norm{e_\lambda}_{L^2(M)}^2,\quad  A_m=\sum_\beta C_{\beta,m},\ B=\sum_\beta B_\beta.
\end{aligned}
\]
	Dividing by $\lambda\log\lambda$ and taking $\lambda\to\infty$
and then $m\to\infty$ proves \eqref{eq:general-littleoh}, uniformly for
$\norm{e_\lambda}_{L^2(M)}=1$.

If $(M,g)$ and $\Sigma$ are real analytic, then $d_g^2$ is analytic
near the diagonal and vanishes there to second order. Hence
\[
\ell(t,s,z,v)^2
=\frac{d_g^2((t,z),(s,z-(t-s)v))}{(t-s)^2},\qquad
\ell(t,t,z,v)^2=|(1,v)|_{g_\Sigma(t,z)}^2>0
\]
extends analytically across $t=s$. Its positive square root $\ell$
is analytic. The analytic implicit function theorem applied to
$\eta=\pm\ell_v$ makes the phases of
Lemma~\ref{upper:kernel-construction} analytic on fixed
neighborhoods of their amplitude supports. Choose the smooth
extensions to agree with these phases on those neighborhoods.
Equation~\eqref{upper:H-analytic} and the same $TT^*$ decomposition give \eqref{logfree}.
\end{proof}

	\section{Proof of Theorem 2}
\label{sec:finite-proof}

The examples in this section are related to earlier counterexamples for endpoint Strichartz estimates in  \cite{MontgomerySmith, Tao}.
A close analogue is Montgomery-Smith's counterexample \cite[Theorem~1]{MontgomerySmith} to the
two-dimensional linear endpoint Strichartz estimate, including the deterministic
argument of Carbery--Hofmann recorded there. The
phase in that argument also appears in \eqref{eq:I0}, and our lacunary
trigonometric profiles make each dyadic frequency contribute a fixed
amount, yielding logarithmic growth. 
Here we realize the two-point phase through spherical distance, use the
exact spherical projector to obtain genuine eigenfunctions, and glue the
local examples into one fixed closed embedded curve of the required
regularity.

Both parts of Theorem~\ref{cor:product-lower} reduce to curves on
$S^3$. Given a $C^1$ embedded closed curve $\Gamma\subset\Sph^2$
and a real eigenfunction $u_{\lambda_\nu}$ on $S^3$ with
$\norm{u_{\lambda_\nu}}_{L^2(S^3)}=1$, let
\[
\Sigma=\Gamma\times Y,\qquad
e_{\lambda_\nu}(x,y)=\operatorname{Vol}(Y)^{-1/2}u_{\lambda_\nu}(x).
\]
The product metric $g$ on $M=S^3\times Y$ gives
\begin{equation}\label{eq:product-identities}
	-\Delta_ge_{\lambda_\nu}=\lambda_\nu^2e_{\lambda_\nu},\qquad
	\norm{e_{\lambda_\nu}}_{L^2(M)}=1,\qquad
	\norm{e_{\lambda_\nu}}_{L^2(\Sigma)}
	=\norm{u_{\lambda_\nu}}_{L^2(\Gamma)}.
\end{equation}
The submanifold $\Sigma$ has codimension two and the same regularity
as $\Gamma$. All restriction bounds are preserved with the same constants.

\subsection{Quadratic forms}
\label{sec:sphere-kernels}

Throughout this section,
$d(x,y)=\arccos(x\cdot y)\in[0,\pi]$ is the round distance on $S^3$.
For $\nu\in\N$, let $\HH_{\nu-1}$ be the space of restrictions to $S^3$
of homogeneous harmonic polynomials of degree $\nu-1$ on $\R^4$
(polynomials $h$ satisfying $h(tx)=t^{\nu-1}h(x)$ and
$\Delta_{\R^4}h=0$). Thus
\begin{equation}\label{eq:frequency}
	-\Delta_{S^3}u_{\lambda_\nu}=(\nu^2-1)u_{\lambda_\nu},\qquad
	\lambda_\nu=\sqrt{\nu^2-1},\qquad u_{\lambda_\nu}\in\HH_{\nu-1}.
\end{equation}
For a compact $C^1$ arc or embedded closed curve $A$, let
$$R_{\nu,A}:\HH_{\nu-1}\to L^2(A),\quad R_{\nu,A}f=f|_A$$ be the restriction operator,
with domain norm $L^2(S^3)$ and induced arclength $\dd\sigma_A$ on $A$.
For a regular parametrization $\gamma$, $\dd\sigma_A=|\gamma'|\dd s$.
Arclength parametrization means $|\gamma'|=1$.  A closed curve below is
parametrized by an injective $C^K$ map
$\gamma:\R/(2\pi\mathbb Z)\to S^3$ with $\gamma'\ne0$.

\begin{lemma}
	\label{lem:adjoint}
	The orthogonal projector onto $\HH_{\nu-1}$ has kernel
	\begin{equation}\label{eq:sphere-kernel}
		\Pi_{\nu-1}(x,y)=\frac{\nu}{2\pi^2}
		\frac{\sin(\nu d(x,y))}{\sin d(x,y)},
	\end{equation}
	with the quotient extended continuously at $d=0,\pi$.
	There is a real $u_{\lambda_\nu}\in\HH_{\nu-1}$ with
	$\norm{u_{\lambda_\nu}}_{L^2(S^3)}=1$ such that for any $f\in L^2(A)$ with $\|f\|_{L^2(A)}=1$, we have $$\norm{u_{\lambda_\nu}}_{L^2(A)}^2\geq \ip{R_{\nu,A}R_{\nu,A}^*f}{f}.$$
\end{lemma}

\begin{proof}
	Rotational invariance gives $\Pi_{\nu-1}(x,y)=F(d(x,y))$.
	The diagonal value is constant, so taking the trace gives
	\[
	2\pi^2F(0)
	=\int_{S^3}\Pi_{\nu-1}(x,x)\dd V(x)
	=\dim\HH_{\nu-1}
	=\binom{\nu+2}{3}-\binom{\nu}{3}
	=\nu^2.
	\]
	Thus $F(0)=\nu^2/(2\pi^2)$. For $0<d<\pi$, the radial Laplacian is
	\[
	\Delta_{S^3}F(d)
	=\frac1{\sin^2d}\bigl(\sin^2d\,F'(d)\bigr)'
	=F''(d)+2\cot d\,F'(d).
	\]
	Since $-\Delta_{S^3}F=(\nu^2-1)F$, we obtain
	\[
	\begin{aligned}
		(F(d)\sin d)''+\nu^2F(d)\sin d
		&=F''(d)\sin d+2F'(d)\cos d+(\nu^2-1)F(d)\sin d\\
		&=\sin d\bigl(F''(d)+2\cot d\,F'(d)+(\nu^2-1)F(d)\bigr)
		=0.
	\end{aligned}
	\]
	Since $F$ is smooth at zero,
	$F(d)\sin d=\frac{\nu}{2\pi^2}\sin(\nu d)$, proving
	\eqref{eq:sphere-kernel}.
	
	The kernel of $R_{\nu,A}R_{\nu,A}^*$ is $\Pi_{\nu-1}$ restricted to
	$A\times A$.  Choose a real $L^2$-normalized eigenfunction $u_{\lambda_\nu}\in\HH_{\nu-1}$ of the
	real nonnegative operator $R_{\nu,A}^*R_{\nu,A}$ for its largest
	eigenvalue.  For $\norm{f}_{L^2(A)}=1$, we have
	\[
	\begin{aligned}
		\norm{u_{\lambda_\nu}}_{L^2(A)}^2
		&=\ip{R_{\nu,A}u_{\lambda_\nu}}{R_{\nu,A}u_{\lambda_\nu}}=\ip{R_{\nu,A}^*R_{\nu,A}u_{\lambda_\nu}}{u_{\lambda_\nu}}\\
		&=\norm{R_{\nu,A}^*R_{\nu,A}}_{\HH_{\nu-1}\to\HH_{\nu-1}}=\norm{R_{\nu,A}R_{\nu,A}^*}_{L^2(A)\to L^2(A)}\\
		&\geq\ip{R_{\nu,A}R_{\nu,A}^*f}{f}.
	\end{aligned}
	\]
\end{proof}

\subsection{Logarithmic growth}
\label{sec:profiles}
The construction below is motivated by  the deterministic argument of Carbery–Hofmann
recorded in \cite[Section 2]{MontgomerySmith}. 
All functions in $C_c^\infty((0,1))$ below are extended by zero to $\R$.
Fix nonnegative $\phi,\chi\in C_c^\infty((0,1))$ such that
\[
\norm{\phi}_{L^2(\R)}=1,\qquad
\min_{[1/4,3/4]}\phi>0,\qquad
\chi=1\quad\text{near }\supp\phi.
\]
For real $q\in C_c^\infty((0,1))$, set
\begin{equation}\label{eq:I0}
	I_0(q)=\int_{0<t<s<1}\frac{\phi(s)\phi(t)}{s-t}
	\sin\left(\frac{(q(s)-q(t))^2}{2(s-t)}\right)\dd t\dd s.
\end{equation}
The integrand is bounded in absolute value by
$\frac12\phi(s)\phi(t)\norm{q'}_{L^\infty(\R)}^2$.

\begin{lemma}\label{lem:profiles}
	There are $c>0$ and real $q_N\in C_c^\infty((0,1))$, $N\geq2$, supported
	in $\supp\chi$, such that
	\begin{equation}\label{eq:profile-bounds}
		I_0(q_N)\geq c\log N,\qquad
		\norm{q_N}_{L^\infty(\R)}\leq C,\qquad
		\norm{q_N^{(m)}}_{L^\infty(\R)}\leq C_mN^{m-1/2}\quad(m\geq1).
	\end{equation}
\end{lemma}

\begin{proof}
	Fix a sufficiently small $\varepsilon>0$ and set
	\begin{equation}\label{eq:deterministic-profile}
		\ell=\lfloor\log_2N\rfloor,\qquad
		q_N(s)=\varepsilon\chi(s)
		\sum_{k=1}^{\ell}2^{-k/2}\cos(8\pi2^ks).
	\end{equation}
	The product rule and geometric sums give
	\[
	\norm{q_N}_{L^\infty(\R)}\leq C\varepsilon,\qquad
	\norm{q_N^{(m)}}_{L^\infty(\R)}
	\leq C_m\varepsilon\sum_{k=1}^{\ell}2^{k(m-1/2)}
	\leq C_mN^{m-1/2},\quad m\geq1.
	\]
	For $s,t\in\supp\phi$ and $h=|s-t|>0$, splitting at $2^kh=1$ gives
	\[
	|q_N(s)-q_N(t)|
	\leq C\varepsilon\sum_{k=1}^{\ell}2^{-k/2}\min\{2^kh,1\}
	\leq C\varepsilon\sqrt h.
	\]
	Thus the sine argument in \eqref{eq:I0} lies in $[0,1]$ for small
	$\varepsilon$, and $\sin x\geq x/2$ on this interval.
	For $0<h<1/4$, $\chi(t)=\chi(t+h)=1$ on $[1/4,1/2]$.
	Orthogonality and $u=4\pi2^kh$ give
	\[
	\begin{aligned}
		\int_{1/4}^{1/2}|q_N(t+h)-q_N(t)|^2\dd t
		&=\frac{\varepsilon^2}{2}
		\sum_{k=1}^{\ell}2^{-k}\sin^2(4\pi2^kh),\\
		2^{-k}\int_0^{1/4}\frac{\sin^2(4\pi2^kh)}{h^2}\dd h
		&=4\pi\int_0^{\pi2^k}\frac{\sin^2u}{u^2}\dd u\geq c.
	\end{aligned}
	\]
	Using $\min_{[1/4,3/4]}\phi>0$, we conclude that
	\[
	I_0(q_N)\geq c\int_0^{1/4}\frac1{h^2}
	\int_{1/4}^{1/2}|q_N(t+h)-q_N(t)|^2\dd t\dd h
	\geq c\varepsilon^2\ell\geq c\log N.\qedhere
	\]
\end{proof}

\subsection{A local lower bound on the sphere}
\label{sec:transfer}

For real $q\in C_c^\infty((0,1))$, $0<r<1/2$, and $\nu\in\N$ with
$(r/\nu)\norm{q}_{L^\infty(\R)}^2<1$, define
\begin{equation}\label{eq:graph}
	\gamma_{\nu,r,q}(s)=
	\left(\sqrt{1-\frac r{\nu} q(s)^2}\cos(rs),
	\sqrt{1-\frac r{\nu} q(s)^2}\sin(rs),
	\sqrt{\frac r{\nu}}\,q(s),0\right),\qquad 0\leq s\leq1.
\end{equation}
Let $A=\gamma_{\nu,r,q}([0,1])$ and $v(s)=|\gamma_{\nu,r,q}'(s)|$.
Set
\[
f_{\nu,r,q}(\gamma_{\nu,r,q}(s))
=v(s)^{-1/2}\phi(s)e^{i\nu rs},\qquad
\norm{f_{\nu,r,q}}_{L^2(A)}=1,
\]
and write
\begin{equation}\label{eq:normalized-Q}
	\mathcal Q_{\nu,r}(q)
	=\nu^{-1}\ip{R_{\nu,A}R_{\nu,A}^*f_{\nu,r,q}}{f_{\nu,r,q}}.
\end{equation}

\begin{lemma}\label{lem:local-lower}
	There are $a>0$ and $N_0\geq2$ such that
	\begin{equation}\label{eq:local-lower}
		\mathcal Q_{\nu,r}(q_N)\geq a\log N
		\qquad(N\geq N_0,\ 0<r<1/2,\ \nu\in\N,\ \nu r\geq N^2).
	\end{equation}
\end{lemma}

\begin{proof}
	Write $q=q_N$, $\gamma=\gamma_{\nu,r,q}$, and
	$b_s=\sqrt{1-(r/\nu)q(s)^2}$. All constants below are independent
	of $N,r,\nu$. Lemma~\ref{lem:profiles} gives
	\[
	\norm q_{L^\infty(\R)}\leq C,\qquad
	\norm{q'}_{L^\infty(\R)}\leq C\sqrt N.
	\]
	Since $\nu r\geq N^2$, for sufficiently large $N$ we have
	$1/2\leq b_s\leq1$ and
	\[
	b_s-1=-\frac{rq(s)^2}{\nu(1+b_s)}=O(r/\nu),\qquad
	b_s'=-\frac{rq(s)q'(s)}{\nu b_s}.
	\]
	Differentiating \eqref{eq:graph} gives
	\begin{equation}\label{eq:local-speed}
		\frac{v(s)^2}{r^2}
		=b_s^2+\frac{(b_s')^2}{r^2}+\frac{q'(s)^2}{\nu r}
		=b_s^2+\frac{q'(s)^2}{\nu r b_s^2}
		=1+O\!\left(\frac N{\nu r}\right).
	\end{equation}
	Since $b_s>0$ and $0<r<1/2$, the first two coordinates of
	$\gamma(s)$ determine $s$, so $\gamma$ is injective.
	
	For $0\leq t<s\leq1$, let
$h=s-t,\  d=d(\gamma(s),\gamma(t)),\ 
	P(s,t)=(q(s)-q(t))^2/(2h).$
	We have
	\[
	\begin{gathered}
		|q(s)-q(t)|\leq C\sqrt N\,h,\qquad
		|b_sb_t-1|\leq Cr/\nu,\\
		b_s-b_t=-\frac{r(q(s)-q(t))(q(s)+q(t))}{\nu(b_s+b_t)},
		\qquad |b_s-b_t|\leq Cr\sqrt N\,h/\nu.
	\end{gathered}
	\]
	Consequently,
	\[
	4|b_sb_t-1|\sin^2(rh/2)+(b_s-b_t)^2
	\leq C(r^3h^2/\nu
	+r^2Nh^2/\nu^2)
	\leq Cr^3h^2/\nu.
	\]
	The squared Euclidean distance therefore gives
	\begin{equation}\label{eq:chord-transfer}
		\begin{aligned}
			4\sin^2(d/2)
			&=|\gamma(s)-\gamma(t)|^2=4b_sb_t\sin^2(rh/2)
			+\frac r\nu(q(s)-q(t))^2+(b_s-b_t)^2\\
			&=4\sin^2(rh/2)+\frac r\nu(q(s)-q(t))^2
			+O\!\left(\frac{r^3h^2}{\nu}\right).
		\end{aligned}
	\end{equation}
	The exact identity above and \eqref{eq:local-speed} imply
	\[
	crh\leq d\leq\int_t^s v(u)\dd u
	=rh\left(1+O\!\left(\frac N{\nu r}\right)\right)<1.
	\]
	Since $(4\sin^2(u/2))'=2\sin u$, the mean value formula gives
	\[
	|d-rh|
	\leq\frac C{rh}\left[
	\frac r\nu(q(s)-q(t))^2+\frac{r^3h^2}{\nu}\right]
	\leq C\frac{Nh}{\nu}.
	\]
	Taylor's formula now yields
	\[
	2\sin(rh)(d-rh)
	=\frac r\nu(q(s)-q(t))^2
	+O\!\left(\frac{r^3h^2}{\nu}
	+\frac{N^2h^2}{\nu^2}\right).
	\]
	Moreover,
	\[
	hP(s,t)=\tfrac12(q(s)-q(t))^2\leq C,\qquad
	\left|\frac{rh}{\sin(rh)}-1\right|P(s,t)
	\leq Cr^2h^2P(s,t)\leq Cr^2h.
	\]
	Multiplying the Taylor identity by $\nu/(2\sin(rh))$ gives
	\begin{equation}\label{eq:transfer-phase-error}
		\begin{aligned}
			\nu d
			&=\nu rh+\frac{rh}{\sin(rh)}P(s,t)
			+O\!\left(h\left(r^2+\frac{N^2}{\nu r}\right)\right)\\
			&=\nu rh+P(s,t)
			+O\!\left(h\left(r^2+\frac{N^2}{\nu r}\right)\right).
		\end{aligned}
	\end{equation}
	Also, \eqref{eq:local-speed} and the bound for $d-rh$ give
	\begin{equation}\label{eq:transfer-amplitude-bounds}
		\frac{h\sqrt{v(s)v(t)}}{\sin d}
		=\frac{\sqrt{v(s)v(t)}}r\frac{rh}d\frac d{\sin d}
		=1+O\!\left(r^2h^2+\frac N{\nu r}\right).
	\end{equation}
	
	By \eqref{eq:sphere-kernel}, pairing $s>t$ and $s<t$ in
	\eqref{eq:normalized-Q} gives
	\[
	2\pi^2\mathcal Q_{\nu,r}(q)
	=2\int_{0<t<s<1}
	\frac{\sqrt{v(s)v(t)}\phi(s)\phi(t)}{\sin d}
	\sin(\nu d)\cos(\nu rh)\dd t\dd s.
	\]
	Using the zero extensions of $q,\phi$, set
	\[
	J=\int_0^1\int_0^1\frac{\phi(t+h)\phi(t)}h
	\sin\bigl(2\nu rh+P(t+h,t)\bigr)\dd h\dd t.
	\]
	Since $\nu d\leq C\nu rh$, \eqref{eq:transfer-phase-error} gives
	\[
	\begin{gathered}
		|\sin(\nu d)|\leq C\min\{\nu rh,1\},\\
		|\sin(\nu d)-\sin(\nu rh+P)|
		\leq Ch\left(r^2+\frac{N^2}{\nu r}\right),\\
		2\sin(\nu rh+P)\cos(\nu rh)
		=\sin P+\sin(2\nu rh+P).
	\end{gathered}
	\]
	Together with \eqref{eq:transfer-amplitude-bounds}, these formulas yield
	\begin{equation}\label{eq:split-Q}
		\begin{aligned}
			|2\pi^2\mathcal Q_{\nu,r}(q)-I_0(q)-J|&\leq C\int_0^1\left[
			\frac{r^2h^2+N/(\nu r)}h\min\{\nu rh,1\}
			+r^2+\frac{N^2}{\nu r}\right]\dd h\\
			&\leq C\left[r^2+\frac{N^2}{\nu r}
			+\frac N{\nu r}\bigl(1+\log(\nu r)\bigr)\right]\leq C.
		\end{aligned}
	\end{equation}
	
	To bound $J$, note that $0\leq P(t+h,t)\leq CNh$ and 
	\[\partial_hP(t+h,t)
		=\frac{q(t+h)-q(t)}h q'(t+h)-\frac{P(t+h,t)}h,
		\qquad |\partial_hP(t+h,t)|\leq CN.
	\]
	For $0<h\leq(\nu r)^{-1}$, we have
	$|\sin(2\nu rh+P)|\leq C\nu rh$. For $(\nu r)^{-1}\leq h\leq1$,
	\[
	\begin{aligned}
		&\partial_h\left[
		\frac{\phi(t+h)\phi(t)e^{iP(t+h,t)}}h\right]=\phi(t)e^{iP(t+h,t)}
		\left[\frac{\phi'(t+h)}h-\frac{\phi(t+h)}{h^2}
		+\frac{i\phi(t+h)}h\partial_hP(t+h,t)\right],\\
		&\left|\partial_h\left[
		\frac{\phi(t+h)\phi(t)e^{iP(t+h,t)}}h\right]\right|
		\leq C\left(\frac Nh+\frac1{h^2}\right).
	\end{aligned}
	\]
	The amplitude vanishes at $h=1$ and is bounded by $C\nu r$ at
	$h=(\nu r)^{-1}$. Integration by parts against $e^{2i\nu rh}$ gives
	\[
	\begin{aligned}
		|J|
		&\leq C\nu r\int_0^{(\nu r)^{-1}}\dd h
		+\frac C{\nu r}\left[\nu r+
		\int_{(\nu r)^{-1}}^1
		\left(\frac Nh+\frac1{h^2}\right)\dd h\right]\\
		&\leq C\left(1+\frac{N\log(\nu r)}{\nu r}\right)\leq C.
	\end{aligned}
	\]
	Combining this with \eqref{eq:split-Q} and Lemma~\ref{lem:profiles},
	\[
	2\pi^2\mathcal Q_{\nu,r}(q_N)
	\geq I_0(q_N)-|J|-C\geq c\log N-C.
	\]
	Increasing $N_0$ proves \eqref{eq:local-lower}.
\end{proof}

\subsection{A fixed curve of finite regularity}
\label{sec:finite-construction}

\begin{proof}[Proof of Theorem~\ref{cor:product-lower}(i)]
	Fix $K\geq1$. For integers $N_j\geq N_0$ to be chosen, set
	\[
	t_j=2^{-j},\qquad r_j=N_j^{-1},\qquad \nu_j=N_j^{4K},\qquad j\geq1,
	\]
	and
	\begin{equation}\label{eq:finite-piece}
		z_j(t)=\sqrt{\frac{r_j}{\nu_j}}\,
		q_{N_j}\!\left(\frac{t-t_j}{r_j}\right)
		=N_j^{-2K-1/2}q_{N_j}\bigl(N_j(t-t_j)\bigr).
	\end{equation}
	Lemma~\ref{lem:profiles} and the chain rule give
	\begin{equation}\label{eq:finite-derivatives}
		\begin{gathered}
			\norm{z_j}_{L^\infty(\R)}\leq CN_j^{-2K-1/2},\\
			\norm{z_j^{(m)}}_{L^\infty(\R)}
			=N_j^{m-2K-1/2}\norm{q_{N_j}^{(m)}}_{L^\infty(\R)}
			\leq C_mN_j^{2m-2K-1},\qquad 1\leq m\leq K.
		\end{gathered}
	\end{equation}
	Thus $\norm{z_j}_{C^K(\R)}\leq C_KN_j^{-1}$, where
	$\norm f_{C^k(\R)}=\max_{0\leq m\leq k}\|f^{(m)}\|_{L^\infty(\R)}$.
	Choose the integers $N_j$ strictly increasing, depending only on $K$, with
	$N_j\geq N_0$ and $N_j>2^{j+4}\max\{1,C_K\}$. Then
	\begin{equation}\label{eq:finite-gluing-bound}
		r_j<2^{-j-4},\qquad \norm{z_j}_{C^K(\R)}\leq2^{-j-4}.
	\end{equation}
	The supports lie in $(t_j,t_j+r_j)$, and $t_{j+1}+r_{j+1}<t_j$.
	Extend each $z_j$ $2\pi$-periodically and set $z=\sum_{j\geq1}z_j$.
	Since
	\[
	\sum_{j\geq1}\norm{z_j}_{C^K(\R)}
	\leq\sum_{j\geq1}2^{-j-4}=\frac1{16},
	\]
	the series converges in $C^K$, with $\norm z_{C^1(\R)}\leq1/16$ and
	$z^{(m)}(0)=\sum_jz_j^{(m)}(0)=0$ for $0\leq m\leq K$.
	Define
	\begin{equation}\label{eq:global-curve}
		\gamma(t)=\bigl(\sqrt{1-z(t)^2}\cos t,
		\sqrt{1-z(t)^2}\sin t,z(t),0\bigr),\qquad
		\Gamma_K=\gamma\bigl(\R/(2\pi\mathbb Z)\bigr).
	\end{equation}
	Since $\sqrt{1-z(t)^2}>0$, the first two coordinates determine $t$
	modulo $2\pi$. Also,
	\begin{equation}\label{speed}
		|\gamma'(t)|^2
		=1-z(t)^2+\frac{z(t)^2z'(t)^2}{1-z(t)^2}+z'(t)^2
		=1-z(t)^2+\frac{z'(t)^2}{1-z(t)^2}>0.
	\end{equation}
	Thus $\Gamma_K\subset\Sph^2$ is a fixed $C^K$ embedded closed curve.
	
	Disjointness of the supports gives, for $0\leq s\leq1$,
	\[
	z(t_j+r_js)=\sqrt{\frac{r_j}{\nu_j}}\,q_{N_j}(s),\qquad
	\nu_jr_j=N_j^{4K-1}\geq N_j^2.
	\]
	Hence $s\mapsto\gamma(t_j+r_js)$ is a rotation of
	$\gamma_{\nu_j,r_j,q_{N_j}}$.
	Lemmas~\ref{lem:adjoint} and~\ref{lem:local-lower}, applied on this arc,
	give real $u_{\lambda_{\nu_j}}\in\HH_{\nu_j-1}$ with
	$\|u_{\lambda_{\nu_j}}\|_{L^2(S^3)}=1$ and
	\begin{equation}\label{eq:finite-Q}
		\|u_{\lambda_{\nu_j}}\|_{L^2(\Gamma_K)}^2
		\geq \nu_j\mathcal Q_{\nu_j,r_j}(q_{N_j})
		\geq a\nu_j\log N_j=\frac a{4K}\nu_j\log \nu_j.
	\end{equation}
	Since $2\leq\lambda_{\nu_j}<\nu_j$, taking square roots gives
	\begin{equation}\label{eq:finite-lower-final}
		\|u_{\lambda_{\nu_j}}\|_{L^2(\Gamma_K)}
		\geq\sqrt{\frac a{4K}}\,
		\lambda_{\nu_j}^{1/2}\sqrt{\log\lambda_{\nu_j}}.
	\end{equation}
	
	For the upper bound, cover any fixed $C^1$ embedded closed curve $\Gamma$
	by finitely many arclength arcs $A=\gamma(I)$ such that
	\[
	|I|<\pi/4,\qquad
	\sup_{s,t\in I}|\gamma'(s)-\gamma'(t)|\leq\tfrac12.
	\]
	For $s,t\in I$,
	\[
	\begin{gathered}
		\gamma(s)-\gamma(t)=(s-t)\gamma'(t)
		+\int_t^s(\gamma'(u)-\gamma'(t))\dd u,\\
		\tfrac12|s-t|\leq|\gamma(s)-\gamma(t)|
		\leq d(\gamma(s),\gamma(t))\leq|s-t|.
	\end{gathered}
	\]
	Writing $d=d(\gamma(s),\gamma(t))<\pi/4$, \eqref{eq:sphere-kernel} gives
	\[
	|\Pi_{\nu-1}(\gamma(s),\gamma(t))|
	\leq C\nu\frac{\min\{\nu d,1\}}d
	\leq C\min\{\nu^2,\nu/|s-t|\}.
	\]
	Schur's test applied to $R_{\nu,A}R_{\nu,A}^*$ yields, for $\nu\geq3$,
	\[
	\begin{aligned}
		\norm{R_{\nu,A}}_{\HH_{\nu-1}\to L^2(A)}^2
		&=\norm{R_{\nu,A}R_{\nu,A}^*}_{L^2(A)\to L^2(A)}\leq C\sup_{s\in I}\int_I\min\{\nu^2,\nu/|s-t|\}\dd t\\
		&\leq C\left(\nu^2\int_0^{1/\nu}\dd h
		+\nu\int_{1/\nu}^1\frac{\dd h}{h}\right)
		=C\nu(1+\log \nu).
	\end{aligned}
	\]
	Summing the squared restriction bounds over the finite cover gives
	\begin{equation}\label{eq:C1-upper}
		\norm{R_{\nu,\Gamma}}_{\HH_{\nu-1}\to L^2(\Gamma)}^2
		\leq C_\Gamma \nu\log(2+\nu).
	\end{equation}
	Apply this to $\Gamma_K$ and use $\nu/2\leq\lambda_\nu<\nu$ for $\nu\geq3$.
	Finally, equation~\eqref{eq:product-identities} transfers the normalization,
	eigenvalue equation, and both restriction bounds to $M$, proving
	\eqref{eq:product-finite}.
\end{proof}
\begin{remark}
	The curve above already belongs to $C^K\setminus C^{K+1}$.
	Orthogonality in \eqref{eq:deterministic-profile} on $[1/4,1/2]$ gives
	$\norm{q_N^{(K+1)}}_{L^\infty(\R)}\geq c_KN^{K+1/2}$, and hence
	\[
	\norm{z_j^{(K+1)}}_{L^\infty(\R)}\geq c_KN_j\longrightarrow\infty.
	\]
	Thus $z\notin C^{K+1}$. The angular coordinate is a local parameter
	of $\Gamma_K$, so no $C^{K+1}$ reparametrization can remove this loss
	of regularity.
\end{remark}

\subsection{A fixed smooth curve}
\label{sec:no-rate}

\begin{proof}[Proof of Theorem~\ref{cor:product-lower}(ii)]
	Fix $L$ as in part~(ii). For integers $N_j\geq N_0$ to be chosen, set
	\begin{equation}\label{eq:smooth-scales}
		t_j=2^{-j},\qquad r_j=N_j^{-1},\qquad \nu_j=N_j^{4j},\qquad
		z_j(t)=N_j^{-2j-1/2}q_{N_j}\bigl(N_j(t-t_j)\bigr),\qquad j\geq1.
	\end{equation}
	Taking $K=j$ in \eqref{eq:finite-derivatives} gives
	\[
	\norm{z_j^{(m)}}_{L^\infty(\R)}\leq C_mN_j^{2m-2j-1}
	\quad(1\leq m\leq j),\qquad
	\norm{z_j}_{C^j(\R)}\leq C_jN_j^{-1}.
	\]
	Choose strictly increasing integers $N_j\geq N_0$ so that
	\begin{equation}\label{eq:select-loss}
		r_j<2^{-j-4},\qquad
		\norm{z_j}_{C^j(\R)}\leq2^{-j-4},\qquad
		\frac{L(\lambda_{\nu_j})}{\sqrt{\log\lambda_{\nu_j}}}
		\leq\frac1j\sqrt{\frac a{4j}}.
	\end{equation}
	The first two conditions hold if $N_j>2^{j+4}\max\{1,C_j\}$.
	For each fixed $j$, the last holds for all sufficiently large $N_j$,
	since $\lambda_{N^{4j}}\to\infty$ and
	$L(\lambda)/\sqrt{\log\lambda}\to0$.
	
	As in part~(i), the supports lie in disjoint intervals $(t_j,t_j+r_j)$.
	Extend each $z_j$ $2\pi$-periodically and set $z=\sum_{j\geq1}z_j$.
	For every fixed integer $m\geq0$,
	\[
	\sum_{j\geq\max\{1,m\}}\norm{z_j}_{C^m(\R)}
	\leq\sum_{j\geq\max\{1,m\}}\norm{z_j}_{C^j(\R)}
	\leq\sum_{j\geq\max\{1,m\}}2^{-j-4}<\infty.
	\]
	The remaining finitely many terms are smooth. Thus $z\in C^\infty(\R)$,
	$z^{(m)}(0)=\sum_jz_j^{(m)}(0)=0$ for every $m\geq0$, and
	\[
	\norm z_{C^1(\R)}\leq\sum_{j\geq1}\norm{z_j}_{C^j(\R)}\leq\frac1{16}.
	\]
	The injectivity and speed calculation \eqref{speed} in part~(i) show that
	\eqref{eq:global-curve} defines a fixed smooth embedded closed curve
	$\Gamma\subset\Sph^2$.
	
	For $0\leq s\leq1$,
	\[
	z(t_j+r_js)=\sqrt{\frac{r_j}{\nu_j}}\,q_{N_j}(s),\qquad
	\nu_jr_j=N_j^{4j-1}\geq N_j^2.
	\]
	The corresponding arc is a rotation of $\gamma_{\nu_j,r_j,q_{N_j}}$.
	Lemmas~\ref{lem:adjoint} and~\ref{lem:local-lower} give real
	$u_{\lambda_{\nu_j}}\in\HH_{\nu_j-1}$ with
	$\|u_{\lambda_{\nu_j}}\|_{L^2(S^3)}=1$ and
	\[
	\begin{aligned}
		\|u_{\lambda_{\nu_j}}\|_{L^2(\Gamma)}^2
		\geq a\nu_j\log N_j
		=\frac a{4j}\nu_j\log \nu_j\geq\frac a{4j}\lambda_{\nu_j}\log\lambda_{\nu_j}.
	\end{aligned}
	\]
	Taking square roots and using \eqref{eq:select-loss}, we obtain
	\[
	\frac{\|u_{\lambda_{\nu_j}}\|_{L^2(\Gamma)}}
	{\lambda_{\nu_j}^{1/2}L(\lambda_{\nu_j})}
	\geq\sqrt{\frac a{4j}}\,
	\frac{\sqrt{\log\lambda_{\nu_j}}}{L(\lambda_{\nu_j})}
	\geq j.
	\]
	The integers $\nu_j=N_j^{4j}$ strictly increase.
	Taking $\Sigma=\Gamma\times Y$ and the product eigenfunctions as in
 \eqref{eq:product-identities} proves \eqref{eq:product-no-rate}.
\end{proof}

	\appendix
	
	\section{A direct proof in dimension three}
	\label{sec:general}
	 We give a direct proof for the smooth case of Theorem \ref{thm:general} in dimension three by using  the Ricci--Stein estimate \cite{RicciStein}, in the form
	\cite[Proposition~2.1]{Pan}:
	\begin{equation}\label{eq:RS}
		\mathcal H_Pf(s)=\pv\int_\R\frac{e^{iP(s,t)}}{s-t}f(t)\dd t,
		\qquad
		\norm{\mathcal H_P}_{L^2(\R)\to L^2(\R)}\leq C_d,
		\qquad \deg P\leq d.
	\end{equation}
	Here $P$ is a real polynomial in $(s,t)$, and $C_d$ depends only on its
	maximum total degree $d$, not on its coefficients.

\begin{lemma}\label{lem:phase}
		Let $I$ be a bounded interval and
		$b\in C^\infty(\overline I\times\overline I;\R)$. Define
		\begin{equation}\label{eq:sine-kernel}
			S_\lambda f(s)=\int_I
			\frac{\sin\bigl(\lambda(s-t)b(s,t)\bigr)}{s-t}f(t)\dd t,
			\qquad \lambda\geq2,
		\end{equation}
		with diagonal kernel value $\lambda b(s,s)$. For every integer $m\geq2$,
		\begin{equation}\label{eq:sine-estimate}
			\norm{S_\lambda}_{L^2(I)\to L^2(I)}
			\leq C_{m,b,I}+\frac2m\log\lambda.
		\end{equation}
	\end{lemma}
	
	\begin{proof}
		After an affine change of variables, assume $I=(0,1)$.
		Partition $I$ into consecutive intervals $I_j$ of length $\lambda^{-1/m}$,
		with the last possibly shorter.  Let $J_j$
		be the union of $I_j$ and its existing neighbors, and define
		\[
		N_\lambda f(s)=\int_{J_j}
		\frac{\sin\bigl(\lambda(s-t)b(s,t)\bigr)}{s-t}f(t)\dd t,
		\qquad s\in I_j.
		\]
		Since $s\in I_j$, $t\notin J_j$ imply $|s-t|\geq\lambda^{-1/m}$, Schur's test gives
		\begin{equation}\label{eq:far-Schur}
			\norm{S_\lambda-N_\lambda}_{L^2(I)\to L^2(I)}
			\leq2\int_{\lambda^{-1/m}}^1\frac{\dd r}{r}
			=\frac2m\log\lambda.
		\end{equation}
		
		Let $p_j$ be the Taylor polynomial of $b$ of total degree $m-2$
		at the center of $I_j\times I_j$.  On $I_j\times J_j$,
		\[
		\frac{\left|\sin\bigl(\lambda(s-t)b(s,t)\bigr)
			-\sin\bigl(\lambda(s-t)p_j(s,t)\bigr)\right|}{|s-t|}
		\leq\lambda|b-p_j|\leq C_m\lambda^{1/m}.
		\]
		As $|I_j|\leq\lambda^{-1/m}$ and $|J_j|\leq3\lambda^{-1/m}$, Schur's test bounds the
		replacement error by $C_m$.  Moreover,
		\[
		\int_{J_j}\frac{\sin\bigl(\lambda(s-t)p_j(s,t)\bigr)}{s-t}f(t)\dd t
		=\frac{\mathcal H_{\lambda(s-t)p_j}-\mathcal H_{-\lambda(s-t)p_j}}{2i}
		(\ind_{J_j}f)(s).
		\]
		Equation~\eqref{eq:RS}, with $\deg((s-t)p_j)\leq m-1$, and
		$\sum_j\ind_{J_j}\leq3$ therefore give
		\[
		\begin{aligned}
			\norm{N_\lambda f}_{L^2(I_j)}&\leq C_m\norm f_{L^2(J_j)},\\
			\norm{N_\lambda f}_{L^2(I)}^2
			&\leq C_m\sum_j\norm f_{L^2(J_j)}^2
			\leq C_m\norm f_{L^2(I)}^2.
		\end{aligned}
		\]
		Together with \eqref{eq:far-Schur}, this proves \eqref{eq:sine-estimate}.
	\end{proof}
	
		Cover $\Sigma$ by finitely many fixed short arclength arcs $A=\gamma(I)$.
		The function $F(s,t)=d_g(\gamma(s),\gamma(t))^2$ is smooth and satisfies
		$F(t,t)=\partial_1F(t,t)=0$, $\partial_1^2F(t,t)=2$.  Taylor's formula gives
		\begin{equation}\label{eq:signed-distance}
			\begin{gathered}
				d_g(\gamma(s),\gamma(t))=|s-t|b(s,t),\qquad b(s,s)=1,\\
				b(s,t)=\left(\int_0^1(1-u)
				\partial_1^2F(t+u(s-t),t)\dd u\right)^{1/2}>0,
			\end{gathered}
		\end{equation}
		with $b$ smooth after shortening the arcs.
		
		Let $P=\sqrt{-\Delta_g}$ and $\chi=\varrho^2$, where
		$\varrho\in\mathcal S(\R)$ is fixed, real and even, $\varrho(0)=1$,
		and $\supp\widehat\varrho$ is sufficiently small.  Let $K_{\lambda,A}$
		be the kernel of $T_{\lambda,A}T_{\lambda,A}^*$, where
		\[
		T_{\lambda,A}f(s)=(\varrho(\lambda-P)f)(\gamma(s)).
		\]
		The Hadamard parametrix and polar coordinates
		\cite[Section~3, (3.3)]{WangZhang} give
		\begin{equation}\label{eq:radial-kernel}
			K_{\lambda,A}(s,t)
			=\frac{w(s,t)}{2\pi^2}\int_0^\infty
			\chi(\lambda-\tau)\tau\frac{\sin(\tau d)}d\dd\tau
			+O_A(\lambda),\qquad d=d_g(\gamma(s),\gamma(t)),
		\end{equation}
		where $w$ is smooth and $w(s,s)=1$.
		For $d,\tau\geq0$, with quotients interpreted continuously at $d=0$,
		\[
		\left|\partial_\tau\!\left(\tau\frac{\sin(\tau d)}d\right)\right|
		=\left|\frac{\sin(\tau d)}d+\tau\cos(\tau d)\right|\leq2\tau.
		\]
		Thus replacing the radial integral by
		$\lambda\sin(\lambda d)d^{-1}\int_\R\chi(u)\dd u$ has error at most
		\[
		\int_\R|\chi(u)|\,|u|(2\lambda+|u|)\dd u
		+\lambda^2\int_\lambda^\infty|\chi(u)|\dd u\leq C_\chi\lambda.
		\]
		Set $c_\chi=(2\pi^2)^{-1}\int_\R\chi(u)\dd u>0$.
		Since $w/b=1+O_A(|s-t|)$, \eqref{eq:radial-kernel} becomes
		\[
		\begin{aligned}
			K_{\lambda,A}(s,t)
			&=c_\chi\lambda\frac{w(s,t)}{b(s,t)}
			\frac{\sin\bigl(\lambda(s-t)b(s,t)\bigr)}{s-t}+O_A(\lambda)\\
			&=c_\chi\lambda
			\frac{\sin\bigl(\lambda(s-t)b(s,t)\bigr)}{s-t}+O_A(\lambda).
		\end{aligned}
		\]
		All errors are pointwise uniform on $I\times I$.  Schur's test gives
		\begin{equation}\label{eq:sine-reduction}
			T_{\lambda,A}T_{\lambda,A}^*
			=c_\chi\lambda S_\lambda+E_{\lambda,A},\qquad
			\norm{E_{\lambda,A}}_{L^2(I)\to L^2(I)}\leq C_A\lambda,
		\end{equation}
		where $S_\lambda$ is given by \eqref{eq:sine-kernel} with this $b$.
		Lemma~\ref{lem:phase} yields
		\[
		\norm{T_{\lambda,A}}_{L^2(M)\to L^2(A)}^2
		\leq c_\chi\lambda\norm{S_\lambda}_{L^2(I)\to L^2(I)}+C_A\lambda
		\leq C_{A,m}\lambda+\frac{2c_\chi}{m}\lambda\log\lambda.
		\]
		Since $T_{\lambda,A}e_\lambda=e_\lambda$ on $A$, summing over the fixed cover gives
		\[
		\norm{e_\lambda}_{L^2(\Sigma)}^2
		\leq\sum_A\norm{T_{\lambda,A}e_\lambda}_{L^2(A)}^2
		\leq\left(A_m\lambda+\frac Bm\lambda\log\lambda\right)
		\norm{e_\lambda}_{L^2(M)}^2,
		\]
		with $B$ independent of $m$.  Dividing by $\lambda\log\lambda$ and then
		letting $\lambda\to\infty$ and $m\to\infty$ proves the uniform little-o estimate.

\end{document}